\documentclass[12pt,a4paper,twoside]{article}

\usepackage[verbose]{geometry}
\usepackage{Macros}

\usepackage{hyperref}
\hypersetup{%
  pdftoolbar=   true,
  pdfmenubar=   true,
  pdffitwindow= true,
  pdftitle=     {Fibrations of AU-contexts beget fibrations of toposes},
  pdfauthor=    {Sina Hazratpour \& Steven Vickers},
  colorlinks=   true,
  linkcolor=    schoolcol,
  citecolor=    schoolcol, 
  urlcolor=     schoolcol }
  
\usepackage{titling}
\usepackage{fancyhdr}
\makeatletter
\@ifclassloaded{amsart}{}{ %
\renewcommand\footnoterule{%
  \vfill
  \kern-3\p@
  \hrule\@width.4\columnwidth
  \kern2.6\p@}
\renewcommand*{\@makefnmark}{\hbox{\textsuperscript{\normalfont [\LiningNumbers\@thefnmark]}}}
\renewcommand*{\@makefntext}[1]{\parindent 1.0em\noindent\ifdefempty{\@thefnmark}{}{\hb@xt@-1.0em{\hss \normalfont [\LiningNumbers \@thefnmark]}\hspace{1.0em}}#1}
\@addtoreset{footnote}{chapter}
\newcommand{\footpar}[1]{\gdef\@thefnmark{}\@footnotetext{#1}}
}
\makeatother

\newcommand{\strong}[1]{\textbfup{#1}}

\title{Fibrations of AU-contexts beget fibrations of toposes}
\author{
  Sina Hazratpour*
  \\
  \texttt{sinahazratpour@gmail.com}
  \and
  Steven Vickers**\\
  \texttt{s.j.vickers@cs.bham.ac.uk}
}
\date{} %21 August 2018

\begin{document}
\pagenumbering{arabic}
\maketitle
\footpar{\{*,**\} School of Computer Science, University of Birmingham, Birmingham, UK.}

\begin{abstract}
Suppose an extension map $U\maps\thT_1 \to \thT_0$ in the 2-category
$\Con$ of contexts for arithmetic universes satisfies a Chevalley criterion for being an (op)fibration in $\Con$.
If $M$ is a model of $\thT_0$ in an elementary topos $\CS$ with nno,
then the classifier $p\maps\CS[\thT_1/M]\to\CS$ satisfies Johnstone's
criterion for being an (op)fibration in the 2-category $\ETopos$
of elementary toposes (with nno) and geometric morphisms.

Along the way, we provide a convenient reformulation of Johnstone's criterion.
\end{abstract}

%%%%%%%%%%
%%%%%%%%%%
\section{Introduction}
\label{sec:intro}
%%%%%%%%%%
For many special constructions of topological spaces (which for us will be point-free, and generalized in the sense of Grothendieck), a structure-preserving morphism between the presenting structures gives a map between the
corresponding spaces. Two very simple examples are: a function $f\maps X \to Y$ between sets already is a map between the corresponding discrete spaces; and a homomorphism $f\maps K \to L$ between two distributive lattices gives a map \emph{in the opposite direction} between their spectra.
The covariance or contravariance of this correspondence is a fundamental property of the construction.

In topos theory we can relativize this process.
A presenting structure in an elementary topos $\CE$ will give rise to a bounded geometric morphism $p \maps \CF \to \CE$, where $\CF$ is the topos of sheaves over $\CE$ for the space presented by the structure. Then we commonly find that the covariant or contravariant correspondence mentioned above makes every such $p$ an opfibration or fibration in the 2-category of toposes and geometric morphisms.

If toposes are taken as bounded over some fixed base $\baseS$, in the 2-category $\BTopos/\baseS$, then there are often easy proofs got by using the Chevalley criterion
to show that the generic such $p$,
taken over the classifying topos for the relevant presenting structures,
is an (op)fibration.
See~\cite{svw:Gelfand-spectra} for some simple examples of the idea,
though there are still questions of strictness left unanswered there.

However, often there is no natural choice of base topos $\baseS$,
and Johnstone~\cite[B4.4]{johnstone:elephant1} proves (op)fibrational results in $\BTopos$.
These are harder both to state (the Chevalley criterion is not available) and to prove,
but stronger because slicing over $\baseS$ restricts the 2-cells.

In this paper we show how to use the arithmetic universe (AU) techniques of \cite{vickers:au-top} to get simple proofs using the Chevalley criterion of the stronger, base-independent (op)fibration results in $\ETopos$,
the 2-category of elementary toposes with nno,
and arbitrary geometric morphisms.

Our starting point is the following construction in \cite{vickers:au-top},
using the 2-category $\Con$ of AU-sketches in \cite{vickers:sk-au}.
Suppose $U \colon \thT_1 \to \thT_0$ is an extension map in $\Con$,
and $M$ is a model of $\thT_0$ in $\baseS$, an elementary topos with nno.
Then there is a geometric theory $\thT_1/M$, of models of $\thT_1$ whose $\thT_0$-reduct is $M$, and so we get a classifying topos
$p \maps \baseS[\thT_1/M] \to \baseS$.
Our main result (Theorem~\ref{thm:main}) is that --
\begin{center}
if $U$ is an (op)fibration in $\Con$, using the Chevalley criterion,

then $p$ is an (op)fibration is $\ETopos$,
using the Johnstone criterion.
\end{center}

Throughout, we assume that
\emph{all our elementary toposes are equipped with
natural numbers object (nno).}
Without an nno the ideas of generalized space do not go far
(because it is needed in order to get an object classifier),
and AU techniques don't apply.

%%%%%%%%%
%%%%%%%%%
\section{Overview}
\label{sec:overview}
%%%%%%%%%
In \textsection \ref{sec:2-categories of toposes}
we review relevant 2-categories of toposes,
including our new 2-category $\GTopos$ in which the
objects are bounded geometric morphisms.

In \textsection \ref{sec:2-category of contexts} we quickly review the main aspects of the theory of AU-contexts,
our AU analogue of geometric theories in which
the need for infinitary disjunctions in many situations has been
satisfied by a type-theoretic style of sort constructions
that include list objects (and an nno).
The contexts are ``sketches for arithmetic universes''
\cite{vickers:sk-au},
and we review the principal syntactic constructions
on them that are used for continuous maps and 2-cells.

In \textsection \ref{sec:classifying-topos-of-ctx}, we
review the connection between contexts and toposes
as developed in~\cite{vickers:au-top},
along with some new results.
A central construction shows how context extension
maps $U\maps\thT_1\to\thT_0$ can be treated as
\emph{bundles} of generalized spaces:
if $M$ is a point of $\thT_0$
(a model of $\thT_0$ in an elementary topos $\CS$),
then the fibre of $U$ over $M$,
as a generalized space over $\CS$,
is a bounded geometric morphism
$p\maps\CS[\thT_1/M]\to\CS$ that classifies the
models of $\thT_1$ whose $U$-reduct is $M$.
Much of the discussion is about understanding the
universal property of such a classifier in the
setting of $\GTopos$.

We then move in \textsection \ref{sec:strict-internal-fibrations-in-2cat}
and \textsection \ref{sec:Johnstone-crit}
to a review of two styles of definition for (op)fibrations in 2-categories,
which we shall call the \emph{Chevalley} and \emph{Johnstone} criteria.

The standard notions of fibration $p\maps E \to B$ as properties of functors between categories
can be generalized to properties of 1-cells in 2-categories,
but how this may be done depends on the structure available in the 2-category.

The basic notion says that for every morphism $u\maps X \to Y$ in a category $B$,
cartesian lifting gives a functor from the fibre of $p$ over $Y$ to that over $X$,
with some universality conditions that express cartesianness.
When we generalize from $\Cat$ to some other 2-category $\CK$,
the obvious generalization is that $X$ and $Y$ become 1-cells from 1 to $B$,
with $u$ a 2-cell between them.
However, even when $\CK$ has a terminal object,
there may fail to be enough 1-cells from 1 to $B$ to make a satisfactory definition this way.
This is generally the case with 2-categories of toposes.

The crude remedy for this is to consider $X$ and $Y$ as 1-cells
from arbitrary objects $B'$ to $B$,
and this underlies Johnstone's definition for $\BTopos$ in \cite[B4.4]{johnstone:elephant1}.
This definition requires very little structure on $\CK$ other than some
-- not necessarily all -- bipullbacks,
sufficient to have bipullbacks of $p$ along all 1-cells to $B$.
We shall call it the \emph{Johnstone} style of definition of fibration.
This is intricate, because it has to deal with many coherence conditions.

Now suppose $\CK$ has comma objects, which unfortunately $\BTopos$ does not,
so far as we know,
although $\BTopos/\baseS$ and our $\Con$ do.
Then we may have a generic $u$,
a generic 2-cell between 1-cells with codomain $B$,
in which the domain of the 1-cells is the cotensor
$\walkarr\pitchfork B$ of $B$ with the walking arrow category $\walkarr$.
In such a $\CK$, the fibration structure for arbitrary $B'$ and $u$
can be got from generic structure for the generic $u$.
The structure needs to be defined just once, instead of many times for all $B'$.
We shall call this a \emph{Chevalley criterion}.
For ordinary fibrations the idea was attributed to Chevalley by Gray~\cite{gray:fibcofibcat},
and subsequently referred to as the Chevalley criterion by Street~\cite{street:fib-yoneda-2cat}.

Note that, even when we can use the Chevalley style,
there are questions about strictness that we must give precise answers to.
Do we have strict commas or bicommas?
(Is the representing object characterized by isomorphisms of categories or equivalences?)
Is a certain counit of an adjunction an isomorphism (as in~\cite{street:fib-yoneda-2cat})
or an identity (as in \cite{gray:fibcofibcat})?
Our main task in \textsection \ref{sec:strict-internal-fibrations-in-2cat}
is to clarify the 2-categorical structure needed, and the strictness issues,
when we apply the Chevalley criterion in $\Con$.

In \textsection\ref{sec:Johnstone-crit} we shall review the Johnstone criterion.
However, we shall also reformulate it in a way that is simpler but at the same time makes it fairly
painless to mix bounded and unbounded geometric morphisms.
We do this with a $2$-category $\GTopos$ ``of Grothendieck toposes'',
with a 2-functor $\ccod$ to $\ETopos$,
so that the fibre $\GTopos(\CS)$ is equivalent to $\bigslant{\BTopos}{\CS}$.
Our formulation uses the cartesian 1-cells and 2-cells for this 2-functor,
and we review the theory of those,
in its bicategorical form,
in \textsection\ref{sec:cartesianness}.

\textsection\ref{sec:main-results} then provides
the main result, Theorem~\ref{thm:main}.
Suppose $U\maps\thT_1\to\thT_0$ is a
context extension map,
and $p\maps \CS[\thT_1/M] \to \CS$ is a
classifier got as in
\textsection \ref{sec:classifying-topos-of-ctx}.
Then if $U$ is an (op)fibration, so is $p$.

%%%%%%%%
%%%%%%%%
\section{Background: 2-categories of toposes}
\label{sec:2-categories of toposes}
%%%%%%%%
The setting for our main result is the 2-category
$\ETopos$ whose $0$-cells are elementary toposes
(equipped with nno),
whose $1$-cells are geometric morphisms,
and whose $2$-cells are geometric transformations.

However, our concern with generalized spaces means that
we must also take care to deal with \emph{bounded} geometric morphisms.
Recall that a geometric morphism $p\maps \CE \to \CS$ is \textit{bounded}
whenever there exists an object $B \in \CE$
(a \emph{bound} for $p$)
such that every $A\in \CE$ is a subquotient of an object of the form $(p^* I) \times B$ for some $I \in \CS$:
that is one can form following span in $\CE$,
with the left leg a mono and the right leg an epi.
\begin{equation*}
  \begin{tikzcd}[column sep=small]
    & E \arrow[dl,rightarrowtail]
      \arrow[dr,twoheadrightarrow] & \\
    (p^* I) \times B  &&  A
  \end{tikzcd}
\end{equation*}

The significance of this notion can be seen in the
relativized version of Giraud's Theorem
(see~\cite[B3.4.4]{johnstone:elephant1}):
$p$ is bounded if and only if $\CE$ can be got as the topos
of sheaves over an internal site in $\CS$.
(In the original Giraud Theorem, relative to $\Set$,
the bound relates to the small set of generators.)
It follows from this that the bounded geometric morphisms
into $\CS$ can be understood as the generalized spaces,
the Grothendieck toposes, relative to $\CS$.

Bounded geometric morphisms are closed under isomorphism and composition
(see~\cite[B3.1.10(i)]{johnstone:elephant1})
and we get a 2-category $\BTopos$ of elementary toposes,
bounded geometric morphisms, and geometric transformations.
It is a sub-2-category of $\ETopos$, full on 2-cells.

Also~\cite[B3.1.10(ii)]{johnstone:elephant1},
if a bounded geometric morphism $q$ is isomorphic to $pf$,
where $p$ is also bounded, then so too is $f$.
This means that if we are only interested in toposes
bounded over $\CS$,
then we do not have to consider unbounded geometric
morphisms between them.
We can therefore take the
``2-category of generalized spaces over $\CS$''
to be the slice 2-category $\BTopos/\CS$,
where the 1-cells are triangles commuting up to an iso-2-cell.
\cite[B4]{johnstone:elephant1} examines $\BTopos/\CS$
in detail.

For the (op)fibrational results,
\cite[B4]{johnstone:elephant1} reverts to $\BTopos$.
This is appropriate, since the properties hold with respect
to arbitrary geometric transformations,
whereas working in $\BTopos/\CS$ limits the discussion to
those that are identities over $\CS$.

Unbounded geometric morphisms are rarely encountered in practice,
and so it might appear reasonable to stay in $\BTopos$
or $\BTopos/\CS$~\cite[B3.1.14]{johnstone:elephant1}.
However, one notable property of bounded geometric morphisms is that
their bipullbacks along \emph{arbitrary} geometric morphisms
exist in $\ETopos$ and are still bounded~\cite[B3.3.6]{johnstone:elephant1}.
(Note that where~\cite{johnstone:elephant1} says pullback in a 2-category,
it actually means bipullback -- this is explained there in section B1.1.)
Thus for any geometric morphism of base toposes
$f \maps \CS' \to  \CS$,
we have the change of base pseudo-$2$-functor
$f \maps  \BTopos/ \CS \to \BTopos/ \CS'$.
One might say the ``2-category of Grothendieck toposes''
is indexed over $\ETopos_{\iso}$
(where the 2-cells in $\ETopos_{\iso}$ are restricted to isos).
\cite{vickers:au-top} develops this in its use of
AU techniques to obtain base-independent topos results,
and there is little additional effort in allowing change of base
along arbitrary geometric morphisms.
To avoid confronting the coherence issues of indexed 2-categories it takes a fibrational approach,
with a 2-category $\GTopos$ ``of Grothendieck toposes''
fibred (in a bicategorical sense) over $\ETopos_{\iso}$.

We shall take a similar approach,
but note that our 2-category $\GTopos$, which we are about to define,
is \emph{not} the same as that of~\cite{vickers:au-top} --
we allow arbitrary geometric transformations ``downstairs''.
We shall write $\GTopos_{\iso}$ when we wish to refer to the $\GTopos$
of~\cite{vickers:au-top}.

\begin{defn} \label{def:GTop}
  The $2$-category $\GTopos$ is defined as follows.
  We use a systematic ``upstairs-downstairs'' notation with overbars and underbars
  to help navigate diagrams.
\begin{enumerate}[label=(\subscript{\GTopos\ }{{\arabic*}})]
\setcounter{enumi}{-1}
  \item
    0-cells are bounded geometric morphisms $x\maps\xobar\to\xubar$.
  \item
    For ant $0$-cells $x$ and $y$, the $1$-cells from $y$ to $x$ are given by triples $f = \tricell{f}$
    where $\fubar \maps \yubar \to \xubar$
    and $\fobar \maps \yobar \to \xobar$
    are geometric morphisms,
    and the geometric transformation
    $\ftri \maps x \fobar \To \fubar y$ is an isomorphism.
  \item
    If $f$ and $g$ are 1-cells from $y$ to $x$,
    then 2-cells from $f$ to $g$ are of the form
    $\alpha= \lang \alobar,\alubar \rang$
    where $\alobar\maps \fobar \To \gobar$
    and $\alubar\maps \fubar \To \gubar$
    are geometric transformations
    so that the obvious diagram of $2$-cells commutes.
\end{enumerate}
  
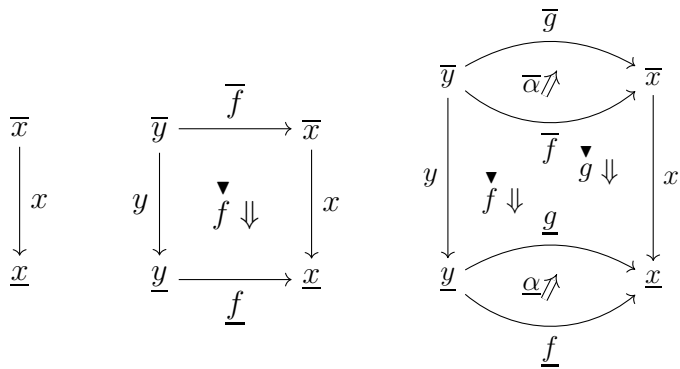
\begin{equation*}
  % 0-cells
    \begin{tikzpicture}[baseline=(c2.base),scale=1,every node/.style={transform shape}]
      \node(c1) at (2,2) {$\xobar$};
      \node(c2) at (2,0) {$\xubar$};
      \draw[->] (c1) to node(g1){} node[right]{$x$} (c2);
    \end{tikzpicture}
    \quad\quad
  % 1-cells
    \begin{tikzpicture}[baseline=(c2.base),scale=1,every node/.style={transform shape}]
      \node(c0) at (0,2) {$\yobar$};
      \node(c1) at (2,2) {$\xobar$};
      \node(c2) at (0,0) {$\yubar$};
      \node(c3) at (2,0) {$\xubar$};
      \node(c4) at (1,-0.7) {} ;
      \node(d5) at (1,1) {$\ftdar$};
      \draw[->] (c2) to node(g1){} node[below]{$\fubar$} (c3);
      \draw[->] (c0) to node(g1){} node[left]{$y$} (c2);
      \draw[->] (c0) to node(g1){} node[above]{$\fobar$} (c1);
      \draw[->] (c1) to node(g1){} node[right]{$x$} (c3);
    \end{tikzpicture}
    \quad\quad
  % 2-cells
    \begin{tikzpicture}[baseline=(c2.base),scale=0.9,every node/.style={transform shape}]
    \node(c0) at (0,1.5) {$\yobar$};
    \node(c1) at (3,1.5) {$\xobar$};
    \node(c2) at (0,-1.5) {$\yubar$};
    \node(c3) at (3,-1.5) {$\xubar$};
    \node(d1) at (0.8,-0.2) {$\ftdar$};
    \node(d2) at (2.2,0.2) {$\gtdar$};
    \node(e1) at (1.3,-2) {};
    \node(e2) at (1.7,-1.2) {};
    \node(e3) at (1.3,1) {};
    \node(e4) at (1.7,1.8) {};
    \draw[->] (c2) to[bend right=40] node(f1){} node[below]{$\fubar$} (c3);
    \draw[->] (c2) to[bend left=30] node(g1){} node[above]{$\gubar$} (c3);
    \draw[->] (c0) to node(y){} node[left]{$y$} (c2);
    \draw[->] (c0) to[bend right=40] node(f2){} node[below]{$\fobar$} (c1);
    \draw[->] (c0) to[bend left=30] node(g2){} node[above]{$\gobar$} (c1);
    \draw[->] (c1) to node(x){} node[right]{$x$} (c3);
    \draw[double,double equal sign distance,-implies, shorten <=2, shorten >=2]
      (e1) to node[left]{$\alubar$} (e2);
    \draw[double,double equal sign distance,-implies, shorten <=2, shorten >=2]
      (e3) to node[left]{$\alobar$} (e4);
    \end{tikzpicture}
  \end{equation*}
  Composition of $1$-cells $k\maps z \to y$ and $f\maps y \to x$ is given by pasting them together,
  more explicitly it is given by
  $f k \eqdef \lang \fobar \oo \kobar, \ftri \odot \ktri ,\fubar \oo \kubar \rang$
  where $\ftri \odot \ktri\eqdef(\fubar \dot \ktri) \oo (\ftri \dot \kobar)$.
  Vertical composition of $2$-cells consists of vertical composition of upper and lower 2-cells.
  Similarly, horizontal composition of $2$-cells consists of horizontal composition
  of upper and lower 2-cells.
  Identity $1$-cells and $2$-cells are defined trivially.
\end{defn}

This is a particular case of our more general
Construction~\ref{cons:display-sub-2-cat}.
$\GTopos$ is $\KK_{\CD}$ when $\KK$ is $\ETopos$ and $\CD$ is the class of bounded geometric morphisms.

Much of our development will turn on the codomain 2-functor
\[
  \ccod \maps \GTopos \to \ETopos
  \text{.}
\]
It is important to note that this codomain functor is \emph{not} a fibration
in any 2-categorical sense,
as it is not well behaved with respect to arbitrary 2-cells in $\ETopos$.
This is easy to see if one takes the point of view of indexed 2-categories,
and considers the corresponding change-of-base functors.
(It becomes a fibration if one restricts the downstairs 2-cells to be isos,
as in~\cite{vickers:au-top}.)
However, it will still be interesting to consider its cartesian 1-cells and 2-cells,
which we do in \textsection\ref{sec:cartesianness}.

%%%%%%%%%
%%%%%%%%%
\section{Background: The 2-category \texorpdfstring{$\Con$}{Con} of contexts}
\label{sec:2-category of contexts}
% \texorpdfstring gives a latex form to use in the section heading itself,
%   and a text form to use in pdf section summaries if we can use the hyperref package.
%%%%%%%%%
The observation underlying~\cite{vickers:sk-au}
is that important geometric theories can be expressed
in coherent logic (no infinite disjunctions),
provided that new sorts can be constructed in a
type-theoretic style that includes free algebra
constructions.
Models can then be sought in any arithmetic universe
(list-arithmetic pretopos),
and that includes any elementary topos with nno;
moreover, the inverse image functors of
geometric morphisms are AU-functors.

In the following table we illustrates some of the differences between the AU approach and toposes. More details about the expressive power of AUs can be found in~\cite{mv:ind-au}.

\begin{figure}[H]
\centering
\begin{tabular} { |c|c|c| }
\hline
  & Arithmetic Universes & Grothendieck toposes \\
\hline
Classifying category & $\AUpres{\thT}$ &  $\CS[\thT]$ \\
\hline
$\thT_1 \to \thT_2$ & $\AUpres{\thT_2} \to \AUpres{\thT_1}$ & $\CS[\thT_1] \to \CS[\thT_2]$   \\
\hline
Base & Base independent &  Base $\CS$ \\
\hline
Infinities & Intrinsic; provided by $\qeql$  & Extrinsic; got from $\CS$     \\
& e.g. $N = \qeql(1)$ & e.g. infinite coproducts  \\
\hline
Results & A single result in AUs & A family of results by    \\
& &  varying $\CS$\\
\hline
\end{tabular}
\end{figure}

The system developed in~\cite{vickers:sk-au}
expresses those geometric theories using
\emph{sketches}.
They are, first of all, finite-limit-finite-colimit
sketches.
Each has an underlying directed graph of \emph{nodes} and \emph{edges},
reflexive to show the identity $s(X)$ for each node $X$,
and with some triangles specified as commutative.
On top of that, certain nodes are specified as being terminal or initial,
and certain cones and cocones are specified as
being for pullbacks or pushouts.
In addition, there is a new notion of
\emph{list universal} to specify parameterized list
objects, together with their empty list and $\qeqlcons$ operations.
From these we can also construct, for example,
$\mathbb{N}$, $\mathbb{Z}$ and $\mathbb{Q}$.

A homomorphism of AU-sketches preserves all the structure:
it maps nodes to nodes, edges to edges, commutativities to commutativities
and universals to universals.

We shall need to restrict the sketches we deal with,
to our \emph{contexts}.
These are built up as \emph{extensions} of the empty sketch $\thone$,
each extension a finite sequence of \emph{simple extension} steps
of the following types: adding a new primitive node, adding a new edge, adding a commutativity, adding a terminal, adding an initial, adding a pullback universal, adding a pushout, and adding a list object.

The following is an example of simple extension by adding a pullback universal.

\begin{ex}
Suppose $\thT$ is a sketch that already contains data in the form of a cospan of edges:
    $\xymatrix{
      {}
        \ar@{->}[r]^{u_1}
      & {}
      & {}
        \ar@{->}[l]_{u_2}
    }$.
Then we can make a simple extension of $\thT$ to $\thT'$
by adding a pullback universal for that cospan,
a cone in the form
\[
  \xymatrix{
        {\mathsf{P}}
          \ar@{.>}[r]^{\mathsf{p}^2}_{\bullet}
          \ar@{.>}[dr]^{\mathsf{p}}
          \ar@{.>}[d]_{\mathsf{p}^1}
        & {}
          \ar@{->}[d]^{u_2}
        \\
        {}
          \ar@{->}[r]_{u_1}^{\bullet}
        & {}
      }
\]
Along with the new universal itself,
we also add a new node $\mathsf{P}$, the pullback;
four new edges (the projections $\mathsf{p}^1,\mathsf{p}^2,\mathsf{p}$
and the identity for $\mathsf{P}$)
and two commutativities
$u_1 \mathsf{p}^1 \skdiag \mathsf{p}$ and
$u_2 \mathsf{p}^2 \skdiag \mathsf{p}$.
\end{ex}

An important feature of extensions is that the subjects of the universals
(for instance, $\mathsf{P}$ and the projections in the above example)
must be \emph{fresh} -- not already in the unextended sketch.
This avoids the possibility of giving a single
node two different universal properties,
and allows the property that every non-strict
model has a canonical strict isomorph.

The next fundamental concept is the notion of \emph{equivalence extension}.
This is an extension that can be expressed in a sequence of
steps for which each introduces structure that must be present, and uniquely,
given the structure in the unextended sketch.
Unlike an ordinary extension, we cannot arbitrarily add nodes,
edges or commutativities -- they must be justified.
Examples of equivalence extensions are to add composite edges;
commutativities that follow from the rules of category theory;
pullbacks, fillins and uniqueness of fillins,
and similarly for terminals, initials, pushouts and list objects;
and inverses of edges that must be isomorphisms by the rules of pretoposes.
Thus the presented AUs for the two contexts are isomorphic.

The previous example, of adding a pullback universal,
is already an equivalence extension.
Another is that of adding a fillin edge.
Suppose we have a pullback universal as above,
and we also have a commutative square
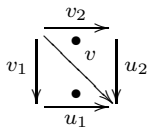
\[
\xymatrix{
        {}
          \ar@{->}[d]_{v_1}
          \ar@{->}[dr]^{v}
          \ar@{->}[r]^{v_2}_{\bullet}
        & {}
          \ar@{->}[d]^{u_2}
        \\
        {}
          \ar@{->}[r]_{u_1}^{\bullet}
        & {}
   }
\]
Then as an equivalence extension we can add a fillin edge
$\xymatrix{{} \ar@{.>}[r]^{w} &{\mathsf{P}}}$,
with commutativities
$\mathsf{p}^1 w \skdiag v_1$ and
$\mathsf{p}^2 w \skdiag v_2$.

Similarly, if we have two fillin candidates with the appropriate commutativities,
then as an equivalence extension we can add a commutativity to say that the fillins are equal.

Any sketch homomorphism between contexts gives a model reduction map
(in the reverse direction),
but those are much too rigidly bound to the syntax to give us a good general
notion of model map.
We seek something closer to geometric morphisms,
and in fact we shall find a notion of \emph{context map}
that captures exactly the strict AU-functors between the corresponding
arithmetic universes $\AUpres{\thT}$.
A \emph{context map} $H\maps\thT_0 \to \thT_1$ is a sketch homomorphism
from $\thT_1$ to some equivalence extension $\thT'_0$ of $\thT_0$.

Each model $M$ of $\thT_0$ gives
-- by the properties of equivalence extensions --
a model of $\thT'_0$,
and then by model reduction along the sketch homomorphism
it gives a model $M\dot H$ of $\thT_1$.

Thus context maps embody a localization by which
equivalence extensions become invertible.
Of course, every sketch homomorphism is, trivially,
a map in the reverse direction.
Context extensions are sketch homomorphisms, and the corresponding maps
backwards are \emph{context extension maps}.
They have some important properties, which we shall see in the next section.

At this point let us introduce the important example of the \emph{hom context}
$\arrw{\thT}$ of a context $\thT$.
We first take two disjoint copies of $\thT$ distinguished by subscripts 0 and 1,
giving two sketch homomorphisms
$i_0,i_1 \maps \thT \to \thT^{\to}$.
Second, for each node $X$ of $\thT$, we adjoin an edge $\theta_X \maps X_0 \to X_1$.
Also, for each edge $u \maps X \to Y$ of $\thT$,
we adjoin a connecting edge $\theta_u \maps X_0 \to Y_1$ together with two commutativities:
\[
  \xymatrix{
    {X_0}
      \ar@{->}[r]^{\theta_X}_{\bullet}
      \ar@{->}[dr]^{\theta_u}
      \ar@{->}[d]_{u_0}
    & {X_1}
      \ar@{->}[d]^{u_1}
    \\
    {Y_0}
      \ar@{->}[r]_{\theta_Y}^{\bullet}
    & {Y_1}
  }
\]
A model of $\arrw{\thT}$ comprises a pair $M_0,M_1$ of models of $\thT$, together with a homomorphism $\theta\maps M_0 \to M_1$.

We define a \emph{2-cell} between maps $H_0,H_1\maps \thT_0 \to \thT_1$
to be a map from $\thT_0$ to $\arrw{\thT_1}$ that composes with the maps
$i_0,i_1\maps \arrw{\thT_1}\to\thT_1$ to give $H_0$ and $H_1$

Finally, an \emph{objective equality} between context maps $H_0$ and $H_1$
is a 2-cell for which the homomorphism between strict models must always be an identity.
This typically arises when a context introduces the same universal construction
twice on the same data.

From these definitions we obtain a 2-category $\Con$ whose
0-cells are contexts,
1-cells are context maps modulo objective equality,
and 2-cells are 2-cells.
It has all PIE-limits (limits constructible from products, inserters, equifiers).
Although it does not possess all (strict) pullbacks of arbitrary maps,
it has all (strict) pullbacks of context extension maps along any other map.

We now list some of most useful example of contexts. For more examples see~\cite[\textsection 3.2]{vickers:sk-au}.

\begin{ex}
\label{ex:pointed-set}
The context $\thob$ has nothing but a single node, $X$, and an identity edge $s(X)$ on $X$.
A model of $\thob$ in an AU (or topos) $\CA$ is a ``set''
in the broad sense of an object of $\CA$,
and so $\thob$ plays the role of the object classifier in topos theory.
There is also a context $\thob[pt]$ which in addition to the generic node $X$
has another node $1$ declared as terminal,
and moreover, it has an edge $x\maps 1 \to X$.
(This is the effect of adding a generic point to the context $\thob$.)
Its models are the pointed sets.
This time we must distinguish between strict and non-strict models.
In a strict model, 1 is interpreted as \emph{the canonical} terminal object.

There is a context extension map $U\maps \thob[pt] \to \thob$
which corresponds to the sketch inclusion in the opposite direction,
sending the generic node in $\thob$ to the generic node in $\thob[pt]$.
As a model reduction, $U$ simply forgets the point.
%\end{ex}

%\begin{ex} \label{ex:pointed-set}
The context $\arrw{\thob}$ comprises two nodes $X_0$ and $X_1$ and their identities, and and an edge $\theta_X\maps X_0 \to X_1$.
A model of $\arrw{\thob}$ in an AU (or topos) $\CA$ is exactly a morphism in $\CA$.
We can define the diagonal context map $\pi_{\thT}\maps \thT \to \arrw{\thT}$ by the opspan $(\id,F)$ of sketch morphisms where $F$ sends edges $\theta_X$ to $s(X)$, $\theta_u$ to $u$ and commutativities to degenerate commutativities of the form $u s(X) \skdiag u$ and $s(Y) u \skdiag u$.
\end{ex}

We outline two more important examples.
We do not have space here to give full details as sketches.
Rather, our aim is to explain why the known geometric theories can be expressed
as contexts.

\newcommand{\Tor}{\mathrm{Tor}}
\begin{ex} \label{ex:torsors}
  Let $\thT_0=[C\colon\mathrm{Cat}]$ be the theory of categories.
  It includes nodes $C_0$ and $C_1$, primitive nodes introduced for the objects of
  objects and of morphisms;
  edges $d_0, d_1\maps C_1\to C_0$ for domain and codomain and an edge for identity morphisms;
  another node $C_2$ for the object of composable pairs and introduced as a pullback;
  an edge $c\maps C_2\to C_1$ for composition;
  and various commutativities for the axioms of category theory.
  The technique is general and would apply to any finite cartesian theory --
  this should be clear from the account in~\cite{pv:PHLCC}.

  Now let us define the extension $\thT_1=[C\colon\mathrm{Cat}][F\colon\Tor(C)]$,
  where $\Tor(C)$ denotes the theory of torsors (flat presheaves) over $C$.
  The presheaf part is expressed by the usual procedure for internal presheaves.
  We declare a node $F_0$ with an edge $p\maps F_0\to C_0$,
  and let $F_1$ be the pullback along $d_0$.
  Then the morphism part of the presheaf defines $xu$ if $d_0(u)=p(x)$,
  and this is expressed by an edge from $F_1$ to $F_0$
  over $d_1$ satisfying various conditions.
  In fact this is another cartesian theory.

  The flatness conditions are not cartesian, but are still expressible using contexts.
  First we must say that $F_0$ is non-empty:
  the unique morphism $F_0\to 1$ is epi,
  in other words the cokernel pair has equal injections.
  Second, if $x,y\in F_0$ then there are $u,v,z$ such that $x=zu$ and $y=zv$.
  Third, if $xu=xv$ then there are $w,z$ such that $x=zw$ and $wu=wv$.
  Again, these can be expressed by saying that certain morphisms are epi.

  Now we have a context extension map $U\maps \thT_1\to\thT_0$,
  which forgets the torsor.

  $\thT_0$ and $\thT_1$, like all contexts, are finite.
  In \textsection\ref{sec:classifying-topos-of-ctx} we shall see how
  for an infinite category $C$ we can still access the infinite theory
  $\Tor(C)$ (infinitely many sorts and axioms, infinitary disjunctions)
  as the ``fibre of $U$ over $C$''.
\end{ex}

\begin{ex} \label{ex:SpecL}
  Let $\thT_0=[L\colon\mathrm{DL}]$ be the finite algebraic theory of distributive lattices,
  a context.
  Now let $\thT_1=[L\colon\mathrm{DL}][F\colon\mathrm{Filt}(L)]$
  be the theory of distributive lattices $L$ equipped with prime filters $F$,
  and let $U\maps\thT_1\to\thT_0$ be the corresponding extension map.
  $\thT_1$ is built over $\thT_0$ by adjoining a node $F$
  with a monic edge $F\to L$,
  and conditions to say that it is a filter
  (contains top and is closed under meet)
  and prime (inaccessible by bottom and join).
  For example, to say that bottom is not in $F$,
  we say that the pullback of $F$ along bottom as edge $1\to L$ is isomorphic to the
  initial object.

  Given a model $L$ of $\thT_0$, the fibre of $U$ over $L$ is its spectrum $\Spec(L)$.
\end{ex}

One central issue for models of sketches is that of \emph{strictness}. The standard sketch-theoretic notion of models is non-strict: for a universal, such as a pullback of some given opspan, the pullback cone can be interpreted as any pullback of the opspan. Contexts give us good handle over strictness. The following result appears in~\cite[Proposition 1]{vickers:au-top}:

\begin{rem}
\label{lift-of-strict-models}
  Let $U\maps\thT_1\to\thT_0$ be an extension map in $\Con$,
  that is to say one got from extending $\thT_0$ to $\thT_1$.
  Suppose in some AU $\CA$ we have a model $M_1$ of $\thT_1$,
  a strict model $M'_0$ of $\thT_0$,
  and an isomorphism $\phi\maps M'_0 \iso M_1 U$.
  \[
    \xymatrix{
      {\thT_1}
        \ar@{->}[d]_{U}
      & {M'_1}
        \ar@{.>}[r]^{\widetilde{\phi}}_{\iso}
        \ar@{.}[d]
      & {M_1}
        \ar@{.}[d]
      \\
      {\thT_0}
      & {M'_0}
        \ar@{->}[r]^{\phi}_{\iso}
      & {M_1 U}
    }
  \]

Then there is a unique model $M'_1$ of $\thT_1$ and isomorphism $\tilde{\phi}\maps M'_1 \iso M_1$ such that

\begin{enumerate}[label=(\roman*)]
  \item
    $M'_1$ is strict,
  \item
    $M'_1 U = M'_0$,
  \item
    $\tilde{\phi} U = \phi$, and
\item
    $\tilde{\phi}$ is equality on all the primitive nodes used in extending $\thT_0$ to $\thT_1$.
\end{enumerate}
We call $M'_1$ the
\emph{canonical strict isomorph} of $M_1$ along $\phi$.
\end{rem}

The fact that we can uniquely lift strict models to strict models as in the remark above will be crucial in \textsection \ref{sec:classifying-topos-of-ctx} and \textsection \ref{sec:main-results}.

%%%%%%%%%%%%%%%%%%%%%%%%%%%
%The connection to contexts
%%%%%%%%%%%%%%%%%%%%%%%%%

\section{Background: Classifying toposes of contexts in \texorpdfstring{$\GTopos$}{GTop}} \label{sec:classifying-topos-of-ctx}

In this part, we shall review how~\cite{vickers:au-top} exploits the fact that,
for any geometric morphism $f\maps\CE\to\CF$ between elementary toposes with nno,
the inverse image functor $\str{f}$ is an AU-functor.
It preserves the finite colimits and finite limits immediately from the definition,
and the preservation of list objects follows quickly from their universal property and
the adjunction of $f$.

By straightforwardly applying $\str{f}$ we transform
a model of $M$ of a context $\thT$ in $\CF$ to a model in $\CE$.
However, we shall be interested in \emph{strict} models,
and $\str{f}$ is in general non-strict as an AU-functor.
For this reason we reserve the notation $\str{f}M$ for the canonical strict isomorph
of the straightforward application, which we write $\str{f}\dot M$.
By this means, the 1-cells of $\ETopos$ act strictly
on the categories of strict  $\thT$-models.
This extends to 2-cells.
If we have $f,g\maps\CE\to\CF$ and $\alpha\maps f \To g$,
then we get a homomorphism $\str{\alpha}M\maps\str{f}M\to\str{g}M$.

It will later be crucial to know how $\str{(-)}$ interacts with transformation of models by context maps.
Given a context map $H\maps \thT_1 \to \thT_0$,
the models $\str{f}(M\dot H)$ and $(\str{f}M)\dot H$ are isomorphic but not always equal.
For instance, take $H\maps \thob[pt] \to \thob$ to be the non-extension context map
that sends the generic node of $\thob$ to the terminal node in $\thob[pt]$,
and $M$ a strict model of $\thob[pt]$.
However,~\cite[Lemma 9]{vickers:au-top} demonstrates that if $H$ is an extension map, then they are indeed equal.

One step further is to investigate the action of $1$-cells and $2$-cells in $\GTopos$ on strict models of context extensions.

\begin{defn}
\label{defn:strict model of context extension}
Let $U\maps \thT_1 \to \thT_0$ be a context extension map
and $p\maps \pobar \to \pubar$ a geometric morphism.

Then a \strong{strict model} of $U$ in $p$
is a pair $M=(\Mobar,\Mubar)$ where $\Mubar$ is a strict $\thT_0$-model in $\pubar$
and $\Mobar$ is a strict $\thT_1$-model in $\pobar$
such that $\Mobar \dot U= \str{p} \Mubar$.

A \strong{$U$-morphism of models}
$\varphi\maps M \to M'$
is a pair $(\fiobar, \fiubar)$ where
$\fiobar\maps\Mobar\to\Mobar'$ and $\fiubar\maps\Mubar\to\Mubar'$
are homomorphisms of $\thT_1$- and $\thT_0$-models
such that $\fiobar \dot U = \str{p} \fiubar$.

Strict $U$-models and $U$-morphisms in $p$ form a category $\model{U}{p}$.
\end{defn}

\begin{cons}
\label{action-of-topos-1-cells-on-models-of-extensions}
Suppose $f\maps q \to p$ is a $1$-cell in $\GTopos$
and let $M$ be a model of $U$ in $p$.
We define a model $\str{f}M$ of $U$ in $q$,
with downstairs part $\str{\fubar}\Mubar$,
as follows.

$\str{\ftri} \Mubar$ provides us with an isomorphism of $\thT_0$-models in $\qobar$
and $(\str{\fobar}\Mobar)\dot U = \str{\fobar}(\Mobar\dot U) = \str{\fobar}\str{p}\Mubar$.
We define the isomorphism
$\str{\ftri} \Mobar \maps \str{\fobar}\Mobar
  \to \str{f}\Mobar$
to be the canonical strict isomorph of
$\str{\fobar}\Mobar$ along $\str{\tri{f}}\Mubar$,
and then $\str{f}M\eqdef (\str{f}\Mobar, \str{\fubar}\Mubar)$
is a strict model of $U$ in $q$.

The construction extends to $U$-model homomorphisms
$\varphi\maps M \to M'$,
as in the diagram on the left.
\begin{equation*}
%\label{action-of-1-cell-on-U}
  \begin{tikzcd}[row sep=scriptsize, column sep=small]
    & \str{\fobar}\Mobar'
      \arrow[rr,dotted,"\str{\ftri} \Mobar'" description]
      \arrow[dd,mapsto]
    & & \str{f}\Mobar'
      \arrow[dd,mapsto]
    \\
    \str{\fobar}\Mobar
      \arrow[ur,"\str{\fobar}\fiobar"]
      \arrow[rr, dotted, crossing over, pos=0.7, "\str{\ftri} \Mobar" description]
      \arrow[dd,mapsto]
    & & \str{f}\Mobar
      \arrow[ur, dotted, "\str{f}\fiobar"]
    \\
    & \str{f}\str{p}\Mubar'
      \arrow[rr, pos=0.25, "\str{\ftri} \Mubar'" description]
    && \str{q} \str{\fubar}\Mubar'
    \\
    \str{\fobar}\str{p}\Mubar
      \arrow[ur,"\str{\fobar}\str{p}\fiubar"]
      \arrow[rr,"\str{\ftri} \Mubar" description]
    & & \str{q} \str{\fubar}\Mubar
      \arrow[ur,swap, "\str{q}\str{\fubar}\fiubar"]
      \arrow[from=uu, mapsto, crossing over]
    \\
  \end{tikzcd}
  \quad
  \begin{tikzcd}[row sep=scriptsize, column sep=small]
    & \str{\gobar}\Mobar
      \arrow[rr,"\str{\gtri} \Mobar" description]
      \arrow[dd,mapsto]
    & & \str{g}\Mobar
      \arrow[dd,mapsto]
    \\
    \str{\fobar}\Mobar
      \arrow[ur,"\str{\alobar}\Mobar"]
      \arrow[rr, crossing over, pos=0.7, "\str{\ftri} \Mobar" description]
      \arrow[dd,mapsto]
    & & \str{f}\Mobar
      \arrow[ur, dotted, "\str{\alpha}\Mobar"]
    \\
    & \str{g}\str{p}\Mubar
      \arrow[rr, pos=0.25, "\str{\gtri} \Mubar" description]
    && \str{q} \str{\gubar}\Mubar
    \\
    \str{\fobar}\str{p}\Mubar
      \arrow[ur,"\str{\alobar}\str{p}\Mubar"]
      \arrow[rr,"\str{\ftri} \Mubar" description]
    & & \str{q} \str{\fubar}\Mubar
      \arrow[ur,swap, "\str{q}\str{\alubar}\Mubar"]
      \arrow[from=uu, mapsto, crossing over]
    \\
  \end{tikzcd}
\end{equation*}
This can be encapsulated in the functor
\[
  \model{U}{f}\maps \model{U}{p} \to \model{U}{q}
  \text{, }
  M \mapsto \str{f}M
  \text{.}
\]
By the properties of the canonical strict isomorph,
it is strictly functorial with respect to $f$.
Furthermore, if $\alpha\maps f \To g$ is a $2$-cell in $\GTopos$,
then the bottom square in the above right-hand diagram commutes
and we define $\str{\alpha}\Mobar$ to be the unique
$\thT_1$-model morphism which completes the top face to a commutative square.
We may also write $\str{f}\Mubar$ and $\str{\alpha}\Mubar$
for $\str{\fubar}\Mubar$ and $\str{\alubar}\Mubar$.
\end{cons}

The upshot is that each $2$-cell $\alpha\maps f \To g$ in $\GTopos$
gives rise to a natural transformation
$\model{U}{\alpha}$ between functors $\model{U}{f}$
and $\model{U}{g}$ and
$(\model{U}{\alpha})(M) = \str{\alpha}M$.
Hence $\model{U}{()}$ is actually a $2$-functor.

\begin{prop}
\label{2-functor-U-Mod}
$\model{U}{()} \maps \op\GTopos \to \Cat$ is a strict $2$-functor.
\end{prop}

A main purpose of~\cite{vickers:au-top} is to explain how a context extension map
$U\maps\thT_1\to\thT_0$ may be thought of as a \emph{bundle},
each point of the base giving rise to a space, its fibre.
In terms of toposes, a point of the base $\thT_0$ is a model $M$ of $\thT_0$ in some
elementary topos $\CS$.
Then the space is a Grothendieck topos over $\CS$,
in other words a bounded geometric morphism.
It should be the classifying topos for a theory $\thT_1/M$ of models of $\thT_1$
that reduce to $M$.

\cite{vickers:au-top} describes $\thT_1/M$ using the approach it calls ``elephant theories'',
namely that set out in~\cite[B4.2.1]{johnstone:elephant1}.
An elephant theory over $\CS$ specifies the category of models of the theory in every
bounded $\CS$-topos $q\maps\CE\to\CS$, together with the reindexing along geometric morphisms.
Then $\thT_1/M$ is defined by letting
$\model{\thT_1/M}{\CE}$ be the category of strict models of $\thT_1$ in $\CE$ that reduce by $U$ to $\str{q}M$.

The extension by which $\thT_1$ was built out of
$\thT_0$ shows that the elephant theory $\thT_1/M$,
while not itself a context,
is geometric over $\CS$ in the sense of~\cite[B4.2.7]{johnstone:elephant1},
and hence has a classifying topos
$p\maps\CS[\thT_1/M]\to\CS$,
with generic model $G$, say.
Its classifying property is that for each bounded $\CS$-topos $\CE$ we have
an equivalence of categories
\[
%\label{classifying-based-topos-eq}
  \Phi\maps \bigslant{\BTopos}{\CS} \ (\CE, \CS[\thT_1/M]) \eqv \model{\thT_1/M}{\CE}
\]
defined as $\Phi(f)\eqdef \str{f} G$.

\begin{ex}
Consider the (unique) context map $!$ from $\thob$
to the empty context $\thone$.
In any elementary topos $\CS$ there is unique model $!$ of $\thone$,
and the classifier for $\thob/!$ is the object
classifier over $\CS$,
the geometric morphism $[\Finset, \CS]\to\CS$
where $\Finset$ here denotes the category
of finite sets as an internal category in $\CS$,
its object of objects being the nno $N$.
The generic model of $\thob$ in $[\Finset,\CS]$
is the inclusion functor
$\inc\maps \Finset \into \Set$.
As an internal diagram it is given by the second
projection of the order $<$ on $N$,
since $\{m\mid m<n\}$ has cardinality $n$.
Given an object $M$ of $\CS$,
The classifying topos for $\thob[pt]/M$ is
the slice topos $\CS/M$.
Hence the classifying topos of $\thob[pt]$ is the slice topos $[\Finset, \CS]/\inc$.
The generic model of $\thob[pt]$ in $[\Finset, \CS]/\inc$ is the pair $(\inc, \pi\maps \inc \to \inc \times \inc)$ where $\Delta$ is the diagonal transformation which renders the diagram below commutative:

\[
  \begin{tikzcd}[row sep=3em]
    \inc
      \arrow{rr}{\Delta}
      \arrow[""{name=foo}]{dr}[swap]{\id} 
    && \inc \times \inc
      \arrow[dl, "\pi_2"]
    \\
    & \inc
  \end{tikzcd}
\]

\end{ex}

So far the discussion of $p$ as classifier has been firmly anchored to $\CS$ and $M$,
but notice that $(G,M)$ is a model of $U$ in $p$.
We now turn to discussing how it fits in more generally with $\model{U}{()}$
by spelling out the properties of $p$ as classifying topos that are shown
in~\cite{vickers:au-top}.
The main result there, Theorem~31, says that $P$ is ``locally representable'' over $Q$
in the following fibration tower.
\newcommand{\PQE}{\GTopos_{\iso}\text{-}U}
\newcommand{\PQC}{\GTopos_{\iso}\text{-}(\thT_0\skext\thT_0)}
\newcommand{\PQB}{\ETopos_{\iso}\text{-}\thT_0}
\[
  \begin{tikzcd}
    \co{(\PQE)}
      \ar[r, "P"]
    & \co{(\PQC)}
      \ar[r, "Q"]
    & \co{(\PQB)}
  \end{tikzcd}
\]
There is a slight change of notation from~\cite{vickers:au-top}.
$\GTopos$ there, unlike ours, restricts the 2-cells to be isomorphisms downstairs.
This is needed to make $P$ and $Q$ fibrations.
To emphasize the distinction we have written $\GTopos_{\iso}$ above.

The objects of $\PQE$ are pairs $(q,N)$ where
$q\maps\qobar\to\qubar$ is a bounded geometric morphism
and $N=(\Nobar,\Nubar)$ is a model of $U$ in $q$.
A 1-cell from $(q_0,N_0)$ to $(q_1,N_1)$
is a triple $(f,f^{-},f_{-})$ such that $f\maps q_0 \to q_1$ in $\GTopos$,
$(f^{-},f_{-})\maps N_0 \to \str{f}N_1$
is a homomorphism of $U$-models,
\emph{and $f_{-}$ is an isomorphism}.
It is $P$-cartesian iff $f^{-}$ too is an isomorphism.
A 2-cell is a 2-cell $\alpha\maps f\To g$ in $\GTopos_{\iso}$ ($\alubar$ an iso)
such that $\str{\alpha}N_1\circ(f^{-},f_{-}) = (g^{-},g_{-})$.

$\PQC$ is similar, but without the $\Nobar$s and $f^{-}$s.

Let us now unravel the local representability.
It says that for each $(\CS,M)$ in $\PQB$ there is a classifier
$(p\maps\CS[\thT_1/M]\to\CS, (G, M))$ in $\PQE$,
where $G$ is the generic model of $\thT_1/M$.

\begin{prop} \label{prop:classifier}
  \cite[Proposition~19]{vickers:au-top}
  The properties that characterize $p$ as classifier are equivalent to the following.
  \begin{enumerate}[label=(\roman*)]
  \item
    For every object $(q, N)$ of $\PQE$,
    1-cell $\fubar\maps \qubar \to \pubar$ in $\ETopos$
    and isomorphism $f_{-}\maps \Nubar \to \str{\fubar}M$,
    there is a $P$-cartesian 1-cell
    $(f,f^{-},f_{-})\maps (q,N)\to (p,(G,M))$
    over $(\fubar,f_{-})$.
    In other words, there is $f$ over $\fubar$
    and an isomorphism ($P$-cartesianness) $f^{-}\maps \Nobar \iso \str{f}G$ over $f_{-}$.
  \item
    Suppose $(f,f^{-},f_{-}), (g,g^{-},g_{-})\maps (q,N) \to (p,(G,M))$ in $\PQE$,
    with $(g,g^{-},g_{-})$ being $P$-cartesian ($g^{-}$ is an iso).
    Suppose also we have $\alubar\maps\gubar\To\fubar$ so that
    $\str{\alubar}M$ commutes with $f_{-}$ and $g_{-}$.
    (Note the reversal of 2-cells compared with \cite[Proposition~19]{vickers:au-top}.
    This is because the fibration tower uses the 2-cell duals $\co{(\PQE)}$ \etc)
    Then $\alubar$ has a unique lift $\alpha\maps g \To f$ such that
    $(\str{\alpha}G)g^{-}=f^{-}$.
  \end{enumerate}
\end{prop}

In the case where we have identity 1-cells and 2-cells downstairs,
it can be seen that this matches the usual characterization of classifier
for $\thT_1/M$ in $\BTopos/\CS$.

Although the properties described above insist on the 2-cells $\alubar$ and
model homomorphisms $f_{-}$ downstairs being isomorphisms,
we shall generalize this in a new result,
Proposition~\ref{prop:non-iso-downstairs}.

We first remark on the construction of finite lax colimits
in the $2$-category $\ETopos$ and more specifically cocomma objects which will be used in our proof.
There is a forgetful $2$-functor $\CU$ from $\op{\ETopos}$
to the $2$-category of categories which sends a topos $\CE$ to its underlying category $\CE$, a geometric morphism $f\maps \CE \to \CF$ to its inverse image part $f^*\maps \CF \to \CE$ and a geometric transformation $\theta\maps f \To g$ to the natural transformation $\theta^*\maps f^* \To g^*$.

The 2-functor $\CU$ transforms colimits in $\ETopos$ to limits in $\Cat$.
This in particular means that the underlying category of a coproduct of toposes,
for instance, is the product of their underlying categories.
The same is true for cocomma objects.
More specifically, for any topos $\CE$,
with cocomma topos $(\id_{\CE} \uparrow \id_{\CE})$ equipped with geometric morphisms
$i_0, i_1\maps \CE \toto (\id_{\CE} \uparrow \id_{\CE})$ and $2$-cell $\theta$ between them,
the data $\lang \str{i}_0, \str{i}_1, \str{\theta} \rang$
specifies the corresponding comma category
$\comma{\id_{\CU(\CE)}}$.
For more details on the construction of cocomma toposes see~\cite[B3.4.2]{johnstone:elephant1}.
Another useful remark is about the relation of topos models of $\arrw{\thT}$ and models of $\thT$.

\begin{lem}
\label{lem:comma-models}
Models of  $\arrw{\thT}$ in a topos $\CE$ are equivalent to
models of $\thT$ in the cocomma topos $(\id_{\CE} \uparrow \id_{\CE})$.
\end{lem}

\begin{prop} \label{prop:non-iso-downstairs}
  Let $U\maps\thT_1\to\thT_0$ be an extension maps of contexts,
  $M$ a strict model of $\thT_0$ in an elementary topos $\CS$,
  and $p\maps\CS[\thT_1/M]\to\CS$ the corresponding classifying topos
  with generic model $G$.

  Let $q\maps\qobar\to\qubar$ be a bounded geometric morphism,
  and let
  $(f_i, f^{-}_i, f_{i-})\maps (q, N_i) \to (p,(G,M))$
  ($i=0,1$) be two $P$-cartesian 1-cells in $\PQE$.

  Suppose $\varphi\maps N_0 \to N_1$ is a homomorphism of $U$-models
  and $\alubar\maps \fubar_0 \To \fubar_1$ is such that the left hand diagram in below commutes.
  Then there exists a unique $2$-cell $\alpha\maps f_0 \To f_1$ over $\alubar$
  such that the right hand diagram commutes.
  \[
    \begin{tikzcd}
      \Nubar_0
        \ar[r, "f_{0-}"]
        \ar[d, "\fiubar"]
      & \str{\fubar_0}M
        \ar[d, "\str{\alubar}M"]
      \\
      \Nubar_1
        \ar[r, "f_{1-}"]
      & \str{\fubar_1}M
    \end{tikzcd}
    \quad
    \begin{tikzcd}
      \Nobar_0
        \ar[r, "f_0^{-}"]
        \ar[d, "\fiobar"]
      & \str{f_0}G
        \ar[d, dotted, "\str{\alpha}G"]
      \\
      \Nobar_1
        \ar[r, "f_1^{-}"]
      & \str{f_1}G
    \end{tikzcd}
  \]
\end{prop}
\begin{proof}
  Note that we do not assume that $\alubar$ and $\fiubar$
  are isomorphisms,
  so $\varphi$ need not be a 1-cell in
  $\GTopos_{\iso}$.
  To get round this, we use cocomma toposes.

  Let $\qubar'= \qubar\uparrow\qubar$ and $\qobar' = \qobar\uparrow\qobar$
  be the two cocomma toposes,
  with bounded geometric morphism $q' \maps \qobar' \to \qubar'$.
  We now have two 1-cells $i_0,i_1\maps q \to q'$ in $\GTopos$,
  equipped with identities for $\itri_0$ and $\itri_1$,
  and a 2-cell $\theta\maps i_0 \To i_1$.
  The pair $\varphi=(\fiobar,\fiubar)$ is a model of $U$ in $q'$.

  The geometric transformation $\alubar$ gives us a geometric morphism
  $\aubar\maps\qubar'\to\CS$,
  with an isomorphism $a_{-}\maps \fiubar \iso \str{\aubar}M$,
  so a 1-cell in $\PQB$.
  This lifts to a $P$-cartesian 1-cell
  $(a, a^{-}, a_{-}) \maps (q', \varphi) \to (p,(G,M))$
  in $\PQE$.
  We now have the following diagrams in $\GTopos$ and $\PQE$.

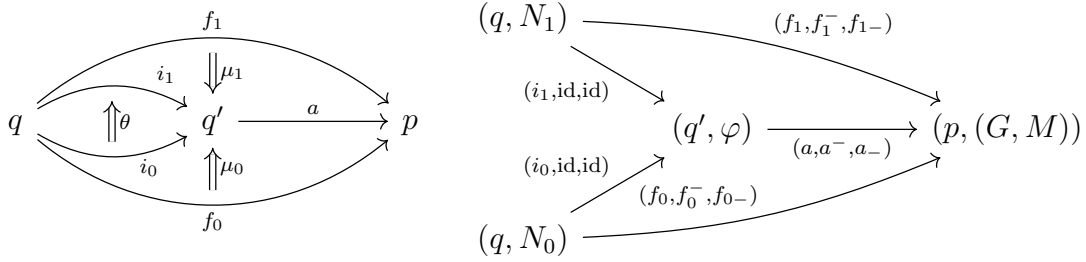
  \[
  \begin{tikzcd}
    q
      \ar[rr, bend left, "i_1" near end, ""{name=ctheta, below}]
      \ar[rr, bend right, "i_0" {below, near end}, ""{name=dtheta, above}]
      \ar[rrrr, bend left=40, "f_1", ""{name=dmu1, below}]
      \ar[rrrr, bend right=40, "f_0" below, ""{name=dmu0, above}]
    && q'
      \ar[rr, "a"]
      \ar[from=dmu0, Rightarrow, "\mu_0" right]
      \ar[from=dmu1, Rightarrow, "\mu_1"]
    && p
    \ar[from=dtheta, to=ctheta, Rightarrow, "\theta" right]
  \end{tikzcd}
  \quad
  \begin{tikzcd}
    (q,N_1)
      \ar[dr, "{(i_1, \id, \id)}" below left]
      \ar[drrr, bend left=10, "{(f_1, f^{-}_1, f_{1-})}"]
    \\
    & (q',\varphi)
      \ar[rr, "{(a,a^{-},a_{-})}" below]
    && (p, (G, M))
    \\
    (q,N_0)
      \ar[ur, "{(i_0, \id, \id)}"]
      \ar[urrr, bend right=10, "{(f_0, f^{-}_0, f_{0-})}"]
  \end{tikzcd}
  \]
  In the right hand diagram all the 1-cells are $P$-cartesian,
  and it follows there are unique iso-2-cells
  $\mu_i\maps (f_1, f^{-}_1, f_{1-}) \to (a,a^{-},a_{-})(i_1, \id, \id)$
  lifting the identity 2-cells downstairs.
  Now by composing $\mu_0$, $a\dot\theta$ and $\mu_1^{-1}$
  we get the required $\alpha$.

  To show uniqueness of the geometric transformation $\alpha$,
  suppose we have another, $\beta$, with the same properties.
  In other words, $\alubar=\betubar$ and
  $\str{\alpha}(G,M)=\str{\beta}(G,M)$.
  We thus get two 1-cells $a,b\maps q' \To p$,
  $a=(f_0,\alpha,f_1)$ and $b=(f_0,\beta,f_1)$.
  We have $\aubar=\bubar$ and $\str{a}(G,M)=\str{b}(G,M)$
  and it follows that there is a unique vertical 2-cell
  $\iota\maps a \To b$ such that $\str{\iota}(G,M)$ is the identity.

  By composing horizontally with $\theta$,
  we can analyse $\iota$ as a pair of 2-cells
  $\iota_\lambda\maps f_\lambda \To f_\lambda$
  ($\lambda=0,1$) such that the following diagram commutes.
  \[
    \begin{tikzcd}
      f_0
        \ar[r, "\iota_0"]
        \ar[d, swap, "\alpha"]
      & f_0
        \ar[d, "\beta"]
      \\
      f_1
        \ar[r, "\iota_1"]
      & f_1
    \end{tikzcd}
  \]
  Now we see that each $\iota_\lambda$ is the the unique vertical 2-cell
  such that $\str{\iota_\lambda}(G,M)$ is the identity,
  so $\iota_\lambda$ is the identity on $f_\lambda$ and $\alpha=\beta$.
\end{proof}

%%%%%%%%
%%%%%%%%
\section{The Chevalley criterion in \texorpdfstring{$\Con$}{Con}}
\label{sec:strict-internal-fibrations-in-2cat}
%%%%%%%%
In~\cite{street:fib-yoneda-2cat}, and later in~\cite{street:fib-in-bicat},
Ross Street develops an elegant algebraic approach to study fibrations, opfibrations,
and two-sided fibrations in 2-categories and bicategories.
In the case of (op)fibrations the 2-category is required to be \emph{finitely complete,}
with strict finite conical limits%
\footnote{
  \ie weighted limits with set-valued weight functors.
  They are ordinary limit as opposed to more general weighted limits.
}
and cotensors with the (free) walking arrow category $\walkarr$.
Given those, it also has strict comma objects.
Then he defined a fibration (opfibration) as a
pseudo-algebra of a certain right (\resp left) slicing 2-monad.
In the case of bicategories they are defined via ``hyperdoctrines'' on bicategories.

For (op)fibrations internal to 2-categories,
he showed~\cite[Proposition 9]{street:fib-yoneda-2cat} that his definition was
equivalent to a Chevalley criterion.
However, for our purposes we prefer to start from the Chevalley criterion
and bypass Street's characterization using pseudoalgebras.

Note also that Street weakened the original Chevalley criterion of Gray,
by allowing the adjunction to have counit an isomorphism.
We shall revert to the original requirement for an identity.

We do not wish to assume existence of all pullbacks since our main 2-category $\Con$
does not have them.
Instead, we assume our 2-categories in this section to have all
finite strict PIE-limits~\cite{pr:pie-limits},
in other words those reducible to Products, Inserters and Equifiers.
This is enough to guarantee existence of all strict comma objects
since for any opspan $A \xraw{f} B \xlaw{g} C$ in a 2-category $\KK$
with (strict) finite PIE-limits,
the comma object $\commacat{f}{g}$ can be constructed as an inserter of
$f \pi_A, g \pi_C \maps A \times C \toto B$.
Moreover, it is a result of~\cite[Lemma 44]{vickers:sk-au}
that $\Con$ has finite PIE-limits.

Pullbacks are not PIE-limits, so sometimes we shall be interested in whether they exist.

\begin{defn}\label{def:carrable}
A $1$-cell $p \maps E \to B$ in a 2-category $\KK$ is \strong{carrable}
whenever a strict pullback of $p$ along any other $1$-cell $f\maps B' \to B$ exists in $\KK$.
As usual, we write $f^\ast p \maps f^\ast E \to B'$ for a chosen pullback of $p$ along $f$.
\end{defn}

\cite{vickers:sk-au} proves that all extension maps in $\Con$ are carrable.

From now on in this section, we assume that $\KK$ is a 2-category with all finite PIE-limits.
Note that for AUs and elementary toposes,
we assume that the structure is given \emph{canonically}
-- this is essential if we are to consider strict models.
For our $\KK$ here we do not assume there are canonical PIE-limits of pullbacks.
Indeed, in $\Con$ (so far as we know)
they do not exist.
1-cells are defined only modulo objective equality,
and the construction of those limits depends on
the choice of representatives of 1-cells.

We first describe the Chevalley criterion in the style of~\cite{street:fib-yoneda-2cat}.
Suppose $B$ is an object of $\KK$, and $p$ is a $0$-cell
in the strict slice $2$-category $\KK/B$.
By the universal property of (strict) comma object $\commacat{B}{p}$,
there is a unique $1$-cell $\Gamma_1 \maps \comma{E} \to \commacat{B}{p} $
with properties $R(p) \Gamma_1 = d_0 \comma{p}$, $\hat{d_1} \Gamma_1 = e_1$,
and $\phi_p \dot \Gamma_1 = p\dot \phi_E$.

\begin{equation}
\label{diag:Chevalley}
\begin{tikzcd}[row sep=3em]
\comma{E}
  \arrow[rrd, bend left= 25, "e_1"]
  \arrow[d, swap, "\comma{p}"]
  \arrow[rd,dotted, swap, "\Gamma_1"] & &
\\
\comma{B}
  \arrow[dr, swap, bend right= 20, "d_0"]
& \commacat{B}{p}
  \arrow[r, "\hat{d_1}"]
%   \arrow[ul, dotted,swap, bend right=15, "\Lambda_1"]
  \arrow[d,swap,"R(p)"]
  \arrow[dr, phantom, bend left=25, ""{name=foo}]
  \arrow[dr, phantom, bend right= 25, ""{name=bar}]
  & E
  \arrow[d,"p"]
\\
 & B
  \arrow[r,swap,"1"]
  \arrow[Rightarrow, phantom, "\fipup", from=bar, to=foo]
  & B
\end{tikzcd}
\end{equation}

\begin{defn}
\label{defn:Chevalley-fibration}
Consider $p$ as above. We call $p$ a \strong{(Chevalley) fibration}
if the $1$-cell $\Gamma_1$ has a right adjoint $\Lambda_1$
with counit $\ep$ an identity in the 2-category $\KK/B$.

Dually one defines (Chevalley) \strong{opfibrations} as $1$-cells $p\maps E \to B$
for which the morphism $\Gamma_0 \maps \comma{E} \to \commacat{p}{B}$ has a left adjoint $\Lambda_0$
with unit $\eta$ an identity.
\end{defn}

Street~\cite{street:fib-yoneda-2cat},
but using isomorphisms for the counits $\ep$
instead of identities,
showed that the Chevalley criterion is equivalent
to a certain pseudoalgebra structure on $p$.
Gray~\cite{gray:fibcofibcat} showed that Chevalley fibrations in the 2-category $\Cat$ of (small) categories correspond to well-known (cloven) Grothendieck fibrations.

In the case where $p$ is carrable, the comma objects $\commacat{p}{B}$ and 
$\commacat{B}{p}$ can be expressed as pullbacks along the two projections from $\comma{B}$ to $B$.
Let us at this point reformulate the fibration property using the notation
as it will appear in $\Con$ when $p$ is an extension map $U\maps \thT_1 \to \thT_0$
-- and using the fact that extension maps are carrable.

Let $\dom, \cod\maps \arrw{\thT_0} \to \thT_0 $ be the domain and codomain context maps corresponding to sketch homomorphisms $i_0,i_1\maps \thT_0 \to \arrw{\thT_0}$.
We define the context extension maps
$\dom^{*}\thT_1 \to \arrw{\thT_0}$
and $\cod^{*}\thT_1 \to \arrw{\thT_0}$
as the pullbacks of $U$ along $\dom$ and $\cod$.
A model of $\dom^*(\thT_1)$ is a pair
$(N, f\maps M_0 \to M_1 )$ where  $f$ is a homomorphism of models of $\thT_0$
and $N$ is a model of $\thT_1$ such that
$N \dot U = M_0$.
Models of $\cod^*(\thT_1)$ are similar,
except that $N \dot U = M_1$.
There are induced context maps
$\Gamma_0\maps \arrw{\thT_1} \to \dom^*(\thT_1)$ and $\Gamma_1\maps \arrw{\thT_1} \to \cod^*(\thT_1)$.
Given a model $f\maps N_0 \to N_1$ of $\arrw{\thT_1}$,
$\Gamma_i$ sends it to
$( N_i, f \dot \arrw{U}\maps N_0 \dot U \to N_1 \dot U )$.

\begin{equation} \label{diag:context-transform}
  \begin{tikzcd}[column sep=1.6em, row sep=1.5em]
    \arrw{\thT_1}
      \arrow[dr,swap, "\Gamma_0"]
      \arrow[rr, swap, bend right=12, "\Gamma_1"]
      \arrow[dddr, bend right= 25, swap, "U^\to"]
    & \top
    & \cod^*(\thT_1)
      \arrow[dr,"\pi_1"]
      \arrow[dd]
      \arrow[ll, dotted, swap, bend right=13, "\Lambda_1"]
      \arrow[rd, phantom, bend right= 25, ""{name=qux}]
    \\
    & \dom^* (\thT_1)
      \arrow[rr, crossing over, bend right=13,
        near end,swap, "\pi_0"]
      \arrow[rr, bend left=1, phantom, ""{name=baz}]
      \arrow[dd, swap, "U_0"]
      \arrow[dr, phantom, bend left= 5,""{name=waldo}]
      \arrow[dr, phantom, bend right= 55, ""{name=fred}]
    & & \thT_1
      \arrow[dd, "U"]
      %\arrow[Rightarrow, from=baz, to=qux, swap, "\alobar" ]
    \\
    && \arrw{\thT_0}
      \arrow[dl, equal, ""{name=foo}]
      \arrow[rd, near start, "\cod"]
      \arrow[dl, phantom, bend left=25, ""{name=alex}]
      \arrow[dr, phantom, bend right=5, ""{name=bar}]
    & &  \\
    & \arrw{\thT_0}
      \arrow[rr, swap, bend right=13, "\dom",
        ""{name=alfred}]
      \arrow[rr, swap, bend left=1, phantom, ""{name=hoax}]
    & & \thT_0
      \arrow[from=uu, crossing over]
      \arrow[Rightarrow, shift right=4pt,
        shorten >= +15pt, pos=0.3, swap, from=hoax,
        to=bar, swap, "\theta_{\thT_0}"]
    \\
  \end{tikzcd}
\end{equation}

\begin{rem}
A consequence of the counit of the adjunction $\Gamma_1 \adj \Lambda_1$ being the identity
is that the adjunction triangle equations are expressed in simpler forms;
we have $\Gamma_1 \dot \eta_1 = \id_{\Gamma_1}$
and $\eta_1\dot\Lambda_1 = \id_{\Lambda_1}$.%
\end{rem}

\begin{rem}
\label{Con:lift-of-identity}
The composite $\Gamma_0 \Lambda_1$
is a $1$-cell from $\cod^* (\thT_1)$ to $\dom^* (\thT_1)$.
Moreover, there is a 2-cell from
$\pi_0\Gamma_0\Lambda_1$ to $\pi_1$
constructed as
$\pi_0 \Gamma_0 \Lambda_1
  \xRaw{\theta_{\thT_1} \Lambda_1}
  \pi_1 \Gamma_1 \Lambda_1
  = \pi_1$.
These two, the 1-cell and the 2-cell,
will appear again as the central structure needed for the Johnstone style of definition
in \textsection \ref{sec:Johnstone-crit}.%
\end{rem}

\begin{ex} \label{ex:pointed-set-opfib}
The context extension $U\maps \thob[pt] \to \thob$
(Example~\ref{ex:pointed-set})
is an extension map with the opfibration property.
\begin{proof}
First we form the pullbacks of the context extension $U$ along the two context maps $\dom$
and $\cod$.
$U_0$ and $U_1$ are $U$ \emph{reindexed} along
$\dom$ and $\cod$:
the same simple extension steps,
but with the data for each transformed by
$\dom$ or $\cod$.
\[
  \begin{tikzcd}[row sep=3em]
    {\dom^{*}(\thob[pt])}
      \arrow{r}{\pi_0}
      \arrow{d}[swap]{U_0}
    & {\thob[pt]}
      \arrow{d}{U}
    \\
    \arrw{\thob}
      \arrow{r}[swap]{\dom}
    &  {\thob}
  \end{tikzcd}
  \quad\quad
  \begin{tikzcd}[row sep=3em]
    {\cod^{*}(\thob[pt])}
      \arrow{r}{\pi_1}
      \arrow{d}[swap]{U_1}
    & {\thob[pt]}
      \arrow{d}{U}
    \\
    \arrw{\thob}
      \arrow{r}[swap]{\cod}
    &  {\thob}
  \end{tikzcd}
\]

$\dom^{*}(\thob[pt])$ is a context with three nodes: a terminal $1$,
primitive nodes $X_0$ and $X_1$,
and edges $x_0\maps 1\to X_0$, $\theta_X\maps X_0 \to X_1$, and identities on the three nodes. $\cod^{*}(\thob[pt])$ is similar,
but with $x_1\maps 1 \to X_1$ instead of $x_0$.

There is, in addition, the arrow context $\arrw{\thob[pt]}$ which consists of all the nodes, edges, and two commutativities
$\theta_X x_0 \sim \theta_x$,
$x_1 \theta_1 \sim \theta_x$
(marked by bullet points) as presented in the following diagram plus identity edges.
\[
    \xymatrix{
      {1_1}
        \ar@{->}[r]^{x_1}_{\bullet}
      & {X_1}
      \\
      {1_0}
        \ar@{->}[r]_{x_0}^{\bullet}
        \ar@{->}[ur]^{\theta_x}
        \ar@{->}[u]^{\theta_1}
      & {X_0}
        \ar@{->}[u]_{\theta_X}
    }
  \]

There are context maps $\Gamma_0$ and $\Gamma_1$ which make the following diagram commute:

\begin{equation*}
\begin{tikzcd}
  & {\cod^{*}\thob[pt]}
    \ar[dr,"\pi_1"]
    \ar[rr, "U_1"]
  && {\thob^\to}
    \arrow[d,"\cod"]
  \\
  \arrw{\thob[pt]}
    \ar[ur,"\Gamma_1"]
    \ar[dr, "\Gamma_0", swap, shift right=1]
  &  & {\thob[pt]}
    \ar[r, "U"]
  &  {\thob}
  \\
  &{\dom^{*}\thob[pt]}
    \ar[ur, swap, "\pi_0"]
    \ar[ul,shift right=1, swap, blue, dotted,
      "\Lambda_0"]
    \ar[rr, swap, "U_0"]
  && {\thob^ \to} \arrow[u, swap, "\dom"]
\end{tikzcd}
\end{equation*}

$\Gamma_0$ is the dual to the sketch morphism
$\dom^*\thob[pt] \to \arrw{\thob[pt]}$
that takes $1$ to $1_0$ and otherwise preserves notation.
$\Gamma_1$ is similar.

More interestingly, $\Gamma_0$ has a left adjoint $\Lambda_0\maps \dom^{*}(\thob[pt]) \to  \arrw{\thob[pt]}$.
For this, $X_0$, $\theta_X$, $X_1$ and $x_0$
in $\arrw{\thob[pt]}$ are interpreted in
$\dom^\ast\thob[pt]$ by the ingredients with the same name,
and $1_0$, $1_1$ by $1$ and
$\theta_1$ by the identity on $1$.
For $\theta_x$ and $x_1$ we need an equivalence extension
of $\dom^\ast\thob[pt]$ got by adjoining the
composite $\theta_X x_0$,
and a commutativity for one of the unit laws
of composition.

It is now obvious that
$\Gamma_0 \Lambda_0 = \id \maps \dom^{*}(\thob[pt]) \to \dom^{*}(\thob[pt])$.
Less obvious, but true in this example, is that
$\Lambda_0 \Gamma_0$ is the identity on $\arrw{\thob[pt]}$.
This follows from the rules for objective equality,
and is essentially because in any strict model
$1_0$ and $1_1$ are both interpreted as the
canonical terminal object,
and $\theta_1$ as the identity on that.
\end{proof}
\end{ex}

We now outline the argument to show that two further examples should be expected to be (op)fibrations.
Details are to appear in~\cite{hazratpour:thesis}.

\begin{ex} \label{ex:SpecLFib}
Let $U\maps\thT_1 \to \thT_0$ be the context extension map
of Example~\ref{ex:SpecL},
for prime filters of distributive lattices.
To show that this is a fibration,
consider a distributive lattice homomorphism $f\maps L_0 \to L_1$.
The map $\Spec(f)\colon \Spec(L_1) \to \Spec(L_0)$ can be expressed using contexts.
It takes a prime filter $F_1$ of $L_1$ to its inverse image $F_0$ under $f$ which is a prime filter of $L_0$.
$f$ restricts (uniquely) to a function from $F_1$ to $F_0$,
and so we get a $\thT_1$-homomorphism
$f'\maps(L_1,F_1)\to(L_0,F_0)$.
The construction so far can all be expressed using AU-structure, and so gives our
$\Lambda_1\maps\cod^\ast(\thT_1)\to\arrw{\thT_1}$.
\[
  \begin{tikzcd}
    (L_0, F_0=f^{-1}(F_1))
      \arrow[r,dotted, "f'"]
      \arrow[mapsto, swap, d, "U"]
    &  (L_1, F_1) \arrow[mapsto, d, "U"]
    \\
    L_0
      \arrow[r, swap, "f"]
    & L_1
  \end{tikzcd}
\]

Aided by the fact that
$\Gamma_1\maps\arrw{\thT_1}
  \to\cod^\ast(\thT_1)$
is given by a sketch homomorphism
(no equivalence extension of $\arrw{\thT_1}$ needed),
we find that $\Gamma_1 \Lambda_1$ is the identity on $\cod^{-1}(\thT_1)$.
The unit $\eta\maps\id\To\Lambda_1\Gamma_1$
of the adjunction is given as follows.
In $\arrw{\thT_1}$ we have a generic
$f\maps(L_0,F_0)\to(L_1,F_1)$,
and clearly $f$ restricted to $F_0$ factors via
$f^{-1}(F_1)$.
Taking this with the identity on $L_1$
gives a $\arrw{\thT_1}$-homomorphism from
$(L_0,F_0)\to(L_1,F_1)$ to
$(L_0,f^{-1}(F_1)\to(L_1,F_1)$,
and hence our $\eta$.
The diagonal equations for the adjunction hold.
\end{ex}

\begin{ex} \label{ex:torsorsOpfib}
Let $U\maps\thT_1 \to \thT_0$ be the context extension map
of Example~\ref{ex:torsors},
for torsors (flat presheaves) of categories.
To show that this is an opfibration,
consider a functor $F\maps C \to D$.

If $T$ is a torsor over $C$,
we must define a torsor $T'= \Tor(F)(T)$ over $D$.
In Example~\ref{ex:torsors} our notation treated the presheaf structure as a right action by $C$ on $T$.
Analogously let us write $D$ as a $C$-$D$-bimodule,
with a right action by $D$ by composition,
and a left action by $C$ by composition after applying $F$.
We define $\Tor(F)(T)$, a $D$-torsor,
as the tensor $T\otimes D$.
Its elements are pairs $(x, f)$
with $x\in T$, $f\in D_1$ and $p(x)=d_0(f)$,
modulo the equivalence relation generated by
$(x, uf)\sim (xu, f)$.
This can be defined using AU structure.
Let us analyse an equation $(x,f)=(x',f')$ in more detail.
It can be expressed as a chain of equations
\[
  (yu,k)\sim^{-1}(y, uk) = (y, u'k')\sim (yu',k')
  \text{,}
\]
each for a quintuple $(k,u,y,u',k')$ with $uk=u'k'$.
Hence the overall equation $(x,f)=(x',f')$
derives from sequences $(k_i)$ ($0\le i\le n$)
and $(u_i),(y_i),(u'_i)$ ($0\le i < n$)
such that $u_i k_i = u'_i k_{i+1}$,
$y_i u'_i = y_{i+1}u_{i+1}$,
$f=k_0$, $x = y_0 u_0$, $f'=k_n$
and $x'=y_{n-1}u'_{n-1}$.
(We are thinking of $k'_i$ as $k_{i+1}$.)
By flatness of $T$ we can replace the $y_i$s
by elements $y v_i$ with
$v_i u'_i = v_{i+1}u_{i+1}$,
$x = y v_0 u_0$ and $x' = y v_{n-1} u'_{n-1}$.

We outline why $\Tor(F)(T)$ is flat (over $D$).
First, it is non-empty, because $T$ is.
If $x\in T$ then $(x,\id_{F(p(x))})\in \Tor(F)(T)$.
Next, suppose $(x,f),(x',f')\in \Tor(F)(T)$.
We can find $y,u,u'$ with $x=yu$ and $x'=yu'$,
and then $(x,f)=(yu,f)=(y,uf)=(y,\id)uf$
and $(x',f')=(\id,y)u'f'$.

Finally, suppose $(x,g)f=(x,g)f'$.
We must find $h,g',y$ such that $hf=hf'$
and $(x,g)=(y,g')h$.
Composing $g'$ and $h$, we can instead look for
$(y,h)=(x,g)$ such that $hf=hf'$.
In fact, we can reduce to the case where $g=\id$.
Suppose, then that we have $(x,f)=(x,f')$.
By the analysis above, we get $y$ and sequences
$(k_i),(u_i),(v_i),(u'_i)$
such that $u_i k_i = u'_i k_{i+1}$,
$v_i u'_i = v_{i+1}u_{i+1}$,
$f=k_0$, $x = y v_0 u_0$, $f'=k_n$
and $x=y v_{n-1}u'_{n-1}$.
Using flatness of $T$ again,
we can assume $v_0 u_0 =v_{n-1}u'_{n-1}$.
Now put $h \eqdef v_0 u_0$,
so $(y,h)=(y, v_0 u_0) = (y v_0 u_0,\id) = (x,\id)$.
Then, as required,
\[
  hf = v_0 u_0 k_0 = v_0 u'_0 k_1 = v_1 u_1 k_1
    = \cdots = v_{n-1} u'_{n-1} k_n = hf'
  \text{.}
\]
Although this reasoning is informal,
its ingredients -- and in particular the reasoning
with finite sequences --
are all present in AU structure.

Once we have $\Tor(F)(T)$ it is straightforward to define to define the function
$T \to \Tor(F)(T)$, $x\mapsto (x,\id)$
that makes a homomorphism of $\thT_1$-models.
Note in particular that the action is preserved:
$xu \mapsto (xu,\id) = (x,u) = (x,\id)u$.
This gives us our $\Lambda_0$,
and $\Gamma_0 \Lambda_0 = \id$.
For the counit of the adjunction,
let $(F,\theta)\maps(C,T)\to(D,T')$ be a
$\thT_1$-homomorphism.
Then $\theta$ factors via $\Tor(F)(T)$ using
$(x,f)\mapsto \theta(x)f$.
This respects the equivalence,
as $\theta(xu)f = \theta(x)F(u)f$
is a condition of $\thT_1$-homomorphisms.
\end{ex}

Note that Example~\ref{ex:pointed-set-opfib} can be got from
Example~\ref{ex:torsorsOpfib} as a pullback.
This is because there is a context map
$\thob\to[C\colon\mathrm{Cat}]$ taking a set $X$
to the discrete category over it.
A torsor over the discrete category is equivalent to an element of $X$.

We conjecture that further examples can be found as follows,
from the basic idea that,
given a style of presentation of spaces,
homomorphisms between presentations can yield
maps between the spaces.
\begin{itemize}
\item
  (Opfibration)
  Let $\thT_0$ be the theory of sets equipped with
  an idempotent relation,
  and $\thT_1$ extend it with a rounded ideal
  \cite{vickers:Infosys}.
\item
  (Opfibration)
  Let $\thT_0$ be the theory of generalized metric spaces,
  and $\thT_1$ extend it with a Cauchy filter
  (point of the localic completion)
  \cite{vickers:LocCompA}.
\item
  (Fibration)
  Let $\thT_0$ be the theory of normal distributive
  lattices, and $\thT_1$ extend it with a rounded
  prime filter
  \cite{svw:Gelfand-spectra}.
  This would generalize Example~\ref{ex:SpecLFib}.
\item
  (Bifibration)
  Let $\thT_0$ be the theory of strongly algebraic information systems,
  and let $\thT_1$ extend it with an ideal
  \cite{vickers:TopCat}.
  This is a special case of Example~\ref{ex:torsorsOpfib}
  -- when the category $C$ is a poset,
  then a torsor is just an ideal --
  and hence would be an opfibration.
  The fibrational nature would come from the fact that
  a homomorphism between two of these information systems
  corresponds to an adjunction between the corresponding
  domains.
\end{itemize}

%%%%%%%%
%%%%%%%%
\section{Remarks on cartesianness for bicategories}
\label{sec:cartesianness}
%%%%%%%%
Our discussion of the Johnstone criterion in \textsection\ref{sec:Johnstone-crit}
will involve a use of the cartesian 1-cells and 2-cells for a 2-functor,
and the present section discusses those.
It is important to note that, although our applications are for 2-functors between 2-categories,
the definitions we use are the ones appropriate to bicategories.

\cite{hermida:fib-2cat} generalizes the notion of fibration to strict 2-functors between strict 2-categories.
His archetypal example of strict 2-fibration is
the 2-category $\Fib$ of Grothendieck fibrations,
fibred over the 2-category of categories via the codomain functor $\ccod \colon \Fib \to \Cat$. Much later~\cite{bakovic:fib-in-tricat} in his talk, and~\cite{buckley:fibred-bicat} in his paper develop these ideas to define  fibration of bicategories.
Borrowing the notions of cartesian $1$-cells and $2$-cells
from their work, we reformulate Johnstone (op)fibrations in terms of the existence of cartesian lifts of $1$-cells and $2$-cells with respect to the codomain functor.
This reformulation will be essential in giving a concise proof of our main result in Theorem~\ref{thm:main}. The Johnstone definition is quite involved and this reformulation effectively organizes the data of various iso-2-cells as part of structure of $1$-cells in the $2$-category $\GTopos$.

We shall examine the cartesian 1-cells and 2-cells for our codomain 2-functor
$\ccod \maps \GTopos \to \ETopos$,
but we might as well do this in the abstract.
We assume for the rest of this section that $\KK$ is a 2-category.

\begin{rem}
\label{rem:recall-bipullback}
We recall that a \strong{bipullback} of an opspan $A \xraw{f} C \xlaw{g} B$ in a 2-category $\KK$ is given by a 0-cell $P$ together with $1$-cells $d_0, d_1$ and an iso-2-cell $\pi \maps fd_0 \To gd_1$ satisfying the following universal properties.

\begin{enumerate}
\item[(BP1)]
  Given another iso-cone $(l_0, l_1, \lambda\colon fl_0 \iso gl_1)$ over $f,g$ (with vertex $X$),
  there exists a $1$-cell $u$ with two iso-2-cells $\gamma_0$ and $\gamma_1$ such that the pasting diagrams below are equal.
\[
  \begin{tikzcd}
      [row sep=large, ampersand replacement=\&]
    X
      \arrow[dr, dashed, "u"]
      \arrow[ddr, bend right= 35, swap, "l_0"]
      \arrow[ddr, phantom, bend right=15,
        ""{name=bar}]
      \arrow[ddr, phantom,
        bend left= 15,""{name=foo}]
      \arrow[drr, bend left =35, "l_1"]
      \arrow[drr, phantom,
        bend right =10,""{name=aap}]
      \arrow[drr, phantom,
        bend left =15, ""{name=noot}]
    \&
      \arrow[Rightarrow, phantom,
        "\iso_{\gamma_0}", swap, shift right=5,
        from=bar, to=foo]
      \arrow[Rightarrow, phantom, shift right=5,
        "\iso_{\gamma_1}" , from=aap, to=noot,swap]
    \\
    \& P
      \arrow[r, "d_1"]
      \arrow[rd, phantom, bend left=25,
        ""{name=foo}]
      \arrow[rd, phantom, bend right=25,
        ""{name=bar}]
      \arrow[d, swap, "d_0"]
    \& B
      \arrow[d, "g"]
      \arrow[Leftarrow, phantom, from= foo,
        to=bar, "\iso_{\pi}"]
    \\
    \& A
      \arrow[r, swap, "f"]
    \& C
  \end{tikzcd}
  \quad = \quad
  \begin{tikzcd}
      [row sep=large, ampersand replacement=\&]
    X
      \arrow[ddr, bend right= 35, swap, "l_0"]
      \arrow[drr, bend left =35, "l_1"]
      \arrow[ddr, phantom, bend left=30,
        ""{name=bar}]
      \arrow[ddr, phantom,
        bend left= 25,""{name=foo}]
    \&
    \\
    \&\& B
      \arrow[d, "g"]
      \arrow[Leftarrow, phantom, shift right=6pt,
        from= foo, to=bar, "\iso_{\lambda}"]
    \\
    \& A
      \arrow[r, swap, "f"]
    \& C
  \end{tikzcd}
\]
\item[(BP2)]
  Given $1$-cells $u,v\colon X \toto P$ and
  $2$-cells $\alpha_i\maps d_i u \To d_i v$
  such that
  \[
    \begin{tikzcd}
      [column sep=large,row sep=large,
        ampersand replacement=\&]
      f d_0 u
        \arrow[r,"f \dot \alpha_0"]
        \arrow[d,swap,"\pi \dot u"]
        \arrow[dr, phantom, bend left=25,
          ""{name=foo}]
        \arrow[dr, phantom, bend right= 25,
          ""{name=bar}] \&
      f d_0 v
        \arrow[d,"\pi \dot v"]
      \\
      g d_1 u
        \arrow[r,swap,"g \dot \alpha_1"]
         \&
      g d_1 v
    \end{tikzcd}
    \text{,}
  \]
  then there is a unique $\beta\maps u \To v$ such that
  each $\alpha_i = d_i\dot\beta$.
\end{enumerate}

The two conditions (BP1) and (BP2) together are equivalent to saying that the functor
\[
\KK(X,P) \xraw{\eqv} \KK(X,A) {\subb{\times}{\KK(X,C)}} \KK(X,B)
\text{,}
\]
obtained from post-composition by the pseudo-cone $\lang d_0, \pi, d_1 \rang$,
is an equivalence of categories.
The right hand side here is an isocomma category.

Note the distinction from pseudopullbacks,
for which the equivalence is an isomorphism of categories.
\end{rem}

\begin{defn}\label{def:bicarrable}
(cf. Definition~\ref{def:carrable}.)
A $1$-cell $x\maps\xobar \to \xubar$ in $\KK$ is \strong{bicarrable} whenever a bipullback of $p$ along any other $1$-cell $\fubar$ exists in $\KK$.
We frequently use the diagram below to represent a chosen such bipullback:
\[
\begin{tikzcd}[column sep=large,row sep=large]
{\strund{f}\xobar}
  \arrow[r,"\fobar"]
  \arrow[d,swap,"\strund{f}x"]
  \arrow[dr, phantom, bend left=35, ""{name=foo}]
  \arrow[dr, phantom, bend right=35, ""{name=bar}]&
{\xobar}
  \arrow[d,"x"]
\\
{\yubar}
  \arrow[r,swap,"\fubar"]
  \arrow[Rightarrow, phantom, "\ftdar",shift left=0.5,from=bar,to=foo] &
\xubar
\end{tikzcd}
\]
where the 2-cell $\ftri$ is an iso-2-cell.

Similarly, we say $x$ is \strong{pseudocarrable} if pseudopullbacks exist.
\end{defn}

Of course, bipullbacks are defined up to equivalence and
the class of bicarrable $1$-cells is closed under bipullback.

An important fact in $\ETopos$ is that all bounded geometric morphisms are bicarrable~\cite[B3.3.6]{johnstone:elephant1}.

\begin{cons}
\label{cons:display-sub-2-cat}
Suppose $\KK$ is a 2-category.
Let $\CD$ be a chosen class of bicarrable $1$-cells in $\KK$,
which we shall call ``display $1$-cells'',
with the following properties.
\begin{itemize}
\item
  Every identity 1-cell is in $\CD$.
\item
  If $x\maps \xobar \to \xubar$ is in $\CD$, and $\fubar\maps\yubar\to\xubar$ in $\KK$,
  then there is some bipullback $y$ of $x$ along $\fubar$ such that $y\in\CD$.
\end{itemize}
We form a $2$-category $\KK_{\CD}$ whose 0-cells
are the elements $x\in \CD$,
and whose 1-cells and 2-cells are taken in exactly
the same manner as for $\GTopos$
(Definition~\ref{def:GTop}),
using elements of $\CD$ for bounded geometric morphisms
and 1-cells and 2-cells in $\KK$
for geometric morphisms and geometric transformations.
\end{cons}

Notice that $\KK_{\CD}$ is a sub-2-category of the 2-category
$\arrK \eqdef \Fun_{ps}(\walkarr, \KK)$, where the latter
consists of (strict) 2-functors,
pseudo-natural transformations and modifications from the interval category (aka free walking arrow category) $\walkarr$.
There is a (strict) 2-functor $\ccod \maps \arrK \to \KK$ which takes 0-cell $x$
(as in above) to its codomain $\xubar$,
a $1$-cell $f$ to $\fubar$ and a $2$-cell $\pair{\alobar}{\alubar}$ to $\alubar$.
The relationship between $\KK$, $\KK_{\CD}$, and $\arrK$ is illustrated
in the following commutative diagram of 2-categories and 2-functors:
\[
\begin{tikzpicture}[commutative diagrams/every diagram]
\matrix[matrix of math nodes, name=m, commutative diagrams/every cell] {
\KK_{\CD} &  &  \arrK \\
& \KK & \\};
\path[commutative diagrams/.cd, every arrow, every label]
(m-1-1) edge[commutative diagrams/hook] (m-1-3)
edge node[swap] {$\ccod$} (m-2-2)
(m-1-3) edge node {$\ccod$} (m-2-2);
\end{tikzpicture}
\]

We now examine cartesian 1-cells and 2-cells of
$\KK_{\CD}$ with respect to $\ccod \maps \KK_{\CD} \to \KK$,
following the definitions of~\cite[3.1]{buckley:fibred-bicat}.
Note that, although we deal only with 2-categories
and 2-functors between them,
we follow the bicategorical definitions,
in which uniqueness appears only at the level of 2-cells.

\begin{defn}
\label{defn:weakly-cart}
Suppose $P \maps \KX \to \KC$ is a $2$-functor.

\begin{enumerate}[label=(\roman*)]
\item A $1$-cell $f\maps y \to x$ in $\KX$ is \strong{cartesian} with respect to $P$ whenever for each 0-cell $w$ in $\KX$ the following commuting square is a bipullback diagram in 2-category $\Cat$ of categories.
\[
\begin{tikzpicture}[commutative diagrams/every diagram]
\matrix[matrix of math nodes, name=m, commutative diagrams/every cell, row sep=3em, column sep= 3em]
{
{\mathscr{X}(w,y)} & {\mathscr{X}(w,x)} \\
{\mathscr{C}(P w, P y)} & {\mathscr{C}(P w, P x)} \\
};
\path[commutative diagrams/.cd, every arrow, every label]
(m-1-1) edge node {$f_{*}$} (m-1-2)
(m-1-2) edge node {$P_{w,x}$} (m-2-2)
(m-1-1) edge node[swap] {$P_{w,y}$} (m-2-1)
(m-2-1) edge node[below] {$P(f)_{*}$} (m-2-2);
\node at (0,0.75) {\tiny bi.p.b.};
\end{tikzpicture}
\]
  This amounts to requiring that, for every object $w$,
  the functor
  \[
    \left\langle P_{w,y}, f_{\ast}\right\rangle \maps
      \mathscr{X}(w,y) \to P(f)_{\ast}\downarrow_{\iso} P_{w,x}
  \]
  should be an equivalence of categories,
  where the category on the right is the isocomma.
  (Note that the image of $\mathscr{X}(w,y)$
  has identities in the squares, not isos.)
\item A 2-cell $\alpha\maps f \To g\maps y \to x$ in $\KX$ is \strong{cartesian} if it is cartesian as a $1$-cell with respect to the functor $P_{yx}\maps \KX(y,x) \to \KC(P y,P x)$.
\end{enumerate}

\end{defn}

The following lemma, which proves certain immediate results about cartesian $1$-cells and $2$-cells, will be handy in the proof of Proposition~\ref{pro:Johnstone-fib-as-cod-fibrant-objects}. The statements are similar to the case of $1$-categorical cartesian morphisms (e.g. in definition of Grothendieck fibrations) with the appropriate weakening of equalities by isomorphisms and isomorphisms by equivalences. They follow straightforwardly from the definition above, however for more details see~\cite{buckley:fibred-bicat}.
In what follows, in keeping with standard nomenclature of theory of categorical fibrations, we regard \textit{vertical} $1$-cells (\resp vertical $2$-cells) as those $1$-cells (\resp $2$-cells) in $\KX$ which are mapped to identity $1$-cells (\resp $2$-cells) in $\KC$ under $P$.

\begin{lem}
\label{lem:closure-properties-of-cart-1,2-cells}
Suppose $P \maps \KX \to \KC$ is a $2$-functor between 2-categories.

\begin{enumerate}[label=(\roman*)]
\item \label{lem:item:closed-under-composition}
  Cartesian $1$-cells (with respect to $P$) are closed under composition and cartesian 2-cells are closed under vertical composition.
\item \label{lem:item:left-closed} Suppose $k \maps w \to y$ and $f \maps y \to x$ are $1$-cells in $\KX$ . If $f$ and $fk$ are cartesian then $k$ is cartesian. The same is true with $2$-cells and their vertical composition.
\item \label{lem:item:identity} Identity $1$-cells and identity $2$-cells are cartesian.
\item\label{lem:item:equivalence} Any equivalence $1$-cell is cartesian.
\item\label{lem:item:iso-2-cell} Any iso $2$-cell is cartesian.
\item\label{lem:item:vert-cart-2-cell} Any vertical cartesian 2-cell is an iso-$2$-cell.
\item \label{lem:item:closed-under-iso-2-cell} Cartesian $1$-cells are closed under isomorphisms:
if $f \iso g$ then $f$ is cartesian if and only if $g$ is cartesian.
\end{enumerate}
\end{lem}

\begin{rem}
\label{rem:explicit-material-of-weak-cart}
\cite[3.1]{buckley:fibred-bicat} also unwinds the definition above
to give a more explicit description of cartesian $1$-cells,
and in particular of the universal properties of pullbacks involved.
A $1$-cell $f\maps y \to x$ is $P$-cartesian if and only if the following hold.
\begin{enumerate}[label=(\roman*)]
\item \label{item:existence-of-explicit-material-of-weak-cart}
For any $1$-cells $g\maps w \to x$ and $\hubar\maps P(w) \to P(y)$ and any iso-2-cell $\alubar \maps Pf \oo \hubar \To Pg$, there exist a $1$-cell $h$ and iso-2-cells $\betubar\maps P(h) \To \hubar$ and $\alpha \maps f h \To g$ such that $P(\alpha) = \alubar \oo (P(f) \dot \betubar)$.
In this situation we call $(h, \betubar)$ a \strong{weak lift} of $\hubar$. If $\betubar$ is the identity 2-cell then we simply call $h$ a \strong{lift} of $\hubar$.
\item \label{item:uniqueness-of-explicit-material-of-weak-cart}
Given any 2-cell $\sigma\colon g \To g' \colon w \toto x$ and $1$-cells $\hubar,\hubar'\maps P(w) \toto P(x) $ and iso-2-cells $\alubar\colon P(f) \oo \hubar \To P(g)$, $\alubar' \colon P(f) \oo \hubar' \To P(g)$ together with any lifts $\pair{h}{\betubar}$ and $\pair{h'}{\betubar'}$ of $\hubar$ and $\hubar'$ respectively, then for any $2$-cell $\delubar\colon \hubar \To \hubar' \colon P(w) \toto P(x)$ satisfying $\alubar' \oo (P(f) \dot \delubar ) = P(\sigma) \oo \alubar$, there exists a unique $2$-cell $\delta \maps h \To h'$ such that $\alpha' \oo (f \dot \delta) = \sigma \oo \alpha$ and $\delubar \dot \betubar = \betubar' \oo P(\delta)$.
\begin{equation}
\label{diag:weak-lift}
\begin{tikzpicture}[scale=0.8, every node/.style={transform shape}]
\node(c1) at (-1,2.5) {$w$};
\node(c2) at (1,-0.5) {$y$};
\node(c3) at (4,-0.5) {$x$};
\node(c4) at (1,-1.2) {} ;
\node(a1) at (2.5, -0.6) {};
\node(a2) at (3.5, 2.2) {};
\draw[->,dashed] (c1) to[bend right=30] node(f2){} node[below left]{$h$} (c2);
\draw[->] (c1) to[bend left=30] node(f4){} node[above]{$g$} (c3);
\draw[->] (c2) to node(g1){} node[below]{$f$} (c3);
\node(d1) at (-1,-4.1) {$P w$};
\node(d2) at (1,-7.1) {$P y$};
\node(d3) at (4,-7.1) {$P x$};
\node(d4) at (1,-2.8) {} ;
\node(e1) at (2.5, -7) {};
\node(e2) at (3.5, -4.6) {};
\draw[->] (d1) to[bend right=70] node(h1){} node[below,xshift=-0.3cm]{$P h$} (d2);
\draw[->] (d1) to[bend left=1] node(h2){} node[above right]{$\hubar$} (d2);
\draw[->] (d1) to[bend left=30] node(h4){} node[above]{$P g$} (d3);
\draw[->] (d2) to node(k1){} node[below]{$P f$} (d3);
\draw[double,double equal sign distance,-implies, dashed, shorten <=2, shorten >=2] (h1) to[bend right=1]  node[below right, pos=0.55]{$\tiny{\betubar}$} (h2);
\draw[|->] (c4) to node(h5){} node[left]{$P$} (d4);
\draw[double,double equal sign distance,-implies,dashed, shorten <=12, shorten >=7] (c2) to[bend right=1]  node[below right, pos=0.6]{$\tiny{\alpha}$} (f4);
\draw[double,double equal sign distance,-implies, shorten <=12, shorten >=7] (d2) to[bend right=1]  node[below right, pos=0.6]{$\tiny{\alubar}$} (h4);
\end{tikzpicture}
\quad \ \ \ \  \quad
\begin{tikzpicture}[scale=0.8,every node/.style={transform shape}]
\node(c1) at (-1,2.5) {$w$};
\node(c2) at (1,-0.5) {$y$};
\node(c3) at (4,-0.5) {$x$};
\node(c4) at (1,-1.2) {} ;
\node(a1) at (2.5, -0.6) {};
\node(a2) at (3.5, 2.2) {};
\draw[->] (c1) to[bend right=20] node(f1){} node[midway, fill=white]{$h$} (c2);
\draw[->] (c1) to[bend left=40] node(f2){} node[midway, fill=white]{$h'$} (c2);
\draw[->] (c1) to[bend left=20] node(f4){} node[midway, fill=white]{$g$} (c3);
\draw[->] (c1) to[bend left=80] node(f5){} node[midway, fill=white]{$g'$} (c3);
\draw[double,double equal sign distance,-implies, shorten <=3, shorten >=3] (f4) to[bend right=1]  node[right, pos=0.3]{$\tiny{\sigma}$} (f5);
\draw[double,double equal sign distance,-implies, dashed, shorten <=3, shorten >=3] (f1) to[bend right=1]  node[above, pos=0.3]{$\tiny{\delta}$} (f2);
\draw[->] (c2) to node(g1){} node[below]{$f$} (c3);
\draw[double,double equal sign distance,-implies, shorten <=7, shorten >=7] (a1) to[bend left=15] node[right, pos=0.6]{$\tiny \alpha'$} (a2);
\node(d1) at (-1,-4.1) {$P w$};
\node(d2) at (1,-7.1) {$P y$};
\node(d3) at (4,-7.1) {$P x$};
\node(d4) at (1,-2.8) {} ;
\node(e1) at (2.5, -7) {};
\node(e2) at (3.5, -4.6) {};
\draw[->] (d1) to[bend right=80] node(h1){} node[below,yshift=-0.3cm]{$Ph$} (d2);
\draw[->] (d1) to[bend right=20] node(h2){} node[midway, fill=white]{$\hubar$} (d2);
\draw[->] (d1) to[bend left=40] node(h3){} node[midway, fill=white]{$\hubar'$} (d2);
\draw[->] (d1) to[bend left=20] node(h4){} node[midway, fill=white]{$P g$} (d3);
\draw[->] (d1) to[bend left=70] node(h5){} node[midway, fill=white]{$P g'$} (d3);
\draw[->] (d2) to node(k1){} node[below]{$P f$} (d3);
\draw[double,double equal sign distance,-implies, shorten <=2, shorten >=2] (h1) to[bend right=1]  node[above, pos=0.3]{$\tiny{\betubar}$} (h2);
\draw[double,double equal sign distance,-implies, shorten <=3, shorten >=3] (h2) to[bend right=1]  node[above, pos=0.3]{$\tiny{\delubar}$} (h3);
\draw[double,double equal sign distance,-implies, shorten <=3, shorten >=3] (h4) to[bend right=1]  node[right, pos=0.3]{$\tiny{P \sigma}$} (h5);
\draw[|->] (c4) to node(h5){} node[left]{$P$} (d4);
\draw[double,double equal sign distance,-implies,shorten <=12, shorten >=8] (c2) to[bend right=1]  node[right, pos=0.5]{$\tiny{\alpha}$} (f4);
\draw[double,double equal sign distance,-implies, shorten <=12, shorten >=8] (d2) to[bend right=1]  node[right, pos=0.5]{$\tiny{\alubar}$} (h4);
\draw[double,double equal sign distance,-implies, shorten <=7, shorten >=7] (e1) to[bend left=15] node[right, pos=0.6]{$\tiny \alubar'$} (e2);
\end{tikzpicture}
\end{equation}
\end{enumerate}
\end{rem}

Also, in elementary terms, a 2-cell $\alpha \colon f_0 \To f_1 \colon y\toto x$ is cartesian iff for any given $1$-cell $e\colon y \to x$ and any 2-cells $\beta \colon e \To f_1$ and $\gamubar \maps P(f_0) \To P(e)$ with $P(\alpha) = P(\beta) \oo \gamubar$, there exists a unique 2-cell $\gamma$ over $\gamubar$ such that $\alpha = \beta \oo \gamma$.
\vspace{10pt}
\[
\xy
(0,15)*+{\bullet}="1";
(0,-15)*+{\bullet}="2";
{\ar@{.>} "1";"2"};
{\ar@/_3.0pc/@{->}_{f_0} "1";"2"};
{\ar@/^3.0pc/@{->}^{f_1} "1";"2"};
(4.5,-5)*+{}="D"; (-4.5,-5)*+{}="C";
(-13,0)*+{}="B"; (12.5,0)*+{}="E";
(0,5)*+{}="A";
%{\ar@{=>} "C";"D"};
{\ar@/_1.0pc/@{=>}^{\alpha} "B";"E"};
%{\ar@{=>} "D";"E"};
{\ar@{==>}^{\gamma} "B";"A"};
{\ar@{=>}^{\beta} "A";"E"};
(0,18)*{y};
(0,-18)*{x};
\endxy
\quad \mapsto \quad
\xy
(0,15)*+{\bullet}="1";
(0,-15)*+{\bullet}="2";
{\ar@{.>} "1";"2"};
{\ar@/_3.0pc/@{->}_{P f_0} "1";"2"};
{\ar@/^3.0pc/@{->}^{P f_1} "1";"2"};
(4.5,-5)*+{}="D"; (-4.5,-5)*+{}="C";
(-13,0)*+{}="B"; (12.5,0)*+{}="E";
(0,5)*+{}="A";
%{\ar@{=>} "C";"D"};
{\ar@/_1.0pc/@{=>}^{P \alpha} "B";"E"};
%{\ar@{=>} "D";"E"};
{\ar@{=>}^{\gamubar} "B";"A"};
{\ar@{=>}^{P \beta} "A";"E"};
(0,18)*{P y};
(0,-18)*{P x};
\endxy
\]

\begin{rem}
The Definition~\ref{defn:weakly-cart} may at first sight seem a bit daunting. Nonetheless the idea behind it is simple; We often think of $\KX$ as bicategory over $\KC$ with richer structures (in practice often times as a fibred bicategory). In this situation, $f\maps y \to x$ being cartesian in $\KX$ means that we can reduce the problem of lifting of any $1$-cell $g$ (with same codomain as $f$) along $f$ (up to an iso-2-cell) to the problem of lifting of $P(g)$ along $P(f)$ in $\KC$ (up to an iso-2-cell). The latter is easier to verify since $\KC$ is a poorer category than $\KX$. The second part of definition says that we also have the lifting of 2-cells along $f$ and the lifted 2-cells are coherent with iso-2-cells obtained from lifting of their respective $1$-cells. This implies the solution to the lifting problem is unique up to a (unique) coherent iso-2-cell.
\end{rem}

We now define a notion that, on the one hand,
conveniently leads to a characterization of when $P$ is a fibration;
but, on the other hand, turns out in the next section to be useful
even when $P$ is not a fibration.

\begin{defn}
\label{def:Johnstone-fib-via-Buckley}
Let $P\maps \KE \to \KB$ be a 2-functor.
We define an object $e$ of $\KE$ to be \strong{fibrational}
iff
\begin{enumerate}[label=(\subscript{B}{{\arabic*}})]
\item \label{def:JfvB:cart-lift-1-cell}
  every $f \maps b' \to b=P(e)$ has a cartesian lift,
\item \label{def:JfvB:cart-lift-2-cell}
  for every 0-cell $e'$ in $\KE$, the functor
  \[
    P_{e',e}\maps \KE(e',e) \to \KB(P(e'), P(e))
  \]
  is a Grothendieck fibration of categories, and
\item \label{def:JfvB:whiskering-cart-2-cell}
  whiskering on the left preserves cartesianness
  of 2-cells in $\KE$ between morphisms with codomain $e$.%
\end{enumerate}
\end{defn}

\begin{rem}
~\cite[Definition 3.1.5]{buckley:fibred-bicat} defines $P$ to be a fibration
  (of bicategories) if
  \begin{enumerate}
  \item
    for every $e$ in $\KE$,
    every $f \maps b' \to b=P(e)$ has a cartesian lift,
  \item
    for all 0-cells $e',e$ in $\KE$, the functor
    \[
      P_{e',e}\maps \KE(e',e) \to \KB(P(e'), P(e))
    \]
    is a Grothendieck fibration of categories, and
  \item
    horizontal composition of 2-cells preserves cartesianness.
  \end{enumerate}
  It is then clear that $P$ is a fibration iff every object of $\KE$
  is fibrational. It is also noteworthy that conditions \ref{def:JfvB:cart-lift-2-cell} and \ref{def:JfvB:whiskering-cart-2-cell} together make the $2$-functor $P_{-,e} \maps \op\KE \to \comma{\Cat}$ lift to $P_{-,e} \maps \op\KE \to \Fib$ for every $e \in \KE$.
\end{rem}

\begin{prop}
\label{pro:cart-1-cell-aka-bipullback}
A $1$-cell in $\KK_{\CD}$ is $\ccod$-cartesian if and only if it is a bipullback square in $\KK$.
\begin{equation}
\label{diag:2-cod-cart}
\begin{tikzpicture}[baseline=(c2.base),scale=1,every node/.style={transform shape}]
\node(c0) at (0,2) {$\yobar$};
\node(c1) at (2,2) {$\xobar$};
\node(c2) at (0,0) {$\yubar$};
\node(c3) at (2,0) {$\xubar$};
\node(c4) at (1,-0.7) {} ;
\node(d2) at (0,-3) {$\yubar$};
\node(d3) at (2,-3) {$\xubar$};
\node(d4) at (1,-2.5) {} ;
\node(d5) at (1,1) {$\ftdar$};
\draw[->] (c2) to node(g1){} node[below]{$\fubar$} (c3);
\draw[->,dashed] (c0) to node(g1){} node[left]{$y$} (c2);
\draw[->,dashed] (c0) to node(g1){} node[above]{$\fobar$} (c1);
\draw[->] (c1) to node(g1){} node[right]{$x$} (c3);
\draw[->] (d2) to node(k1){} node[below]{$\fubar$} (d3);
\draw[|->] (c4) to node(h5){} node[left]{$\ccod$} (d4);
\draw +(.25,1.45) -- +(.55,1.45)  -- +(.55,1.75);
\end{tikzpicture}
\end{equation}
\end{prop}

Before giving the proof there is one step we take to simplify the proof.
\begin{lem}
\label{lem:strict-lift-replacement}
Suppose $\hubar \maps \wubar \to \yubar$ is a $1$-cell in $\KK$. Any weak lift $(h_0, \betubar)$ of $\hubar$ can be replaced by a lift $h$ in which $\betubar$ is replaced by the identity 2-cell. Therefore, conditions (i) and (ii) in Remark~\ref{rem:explicit-material-of-weak-cart} can be rephrased to simpler conditions in which $\betubar$ is the identity 2-cell.
\end{lem}
\begin{proof}
Define $\hobar = \hobar_0$, and $\htri = (\betubar \dot w) \oo \tri{h_0}$:
\[
\begin{tikzpicture}[baseline=(c2.base),scale=1,every node/.style={transform shape}]
\node(c0) at (0,2) {$\wobar$};
\node(c1) at (2,2) {$\yobar$};
\node(c2) at (0,0) {$\wubar$};
\node(c3) at (2,0) {$\yubar$};
\node(d1) at (1,1.3) {$\hztdar$};
\node(e1) at (1.1,-0.7) {};
\node(e2) at (1.1,0) {};
\draw[->] (c0) to node(f2){} node[above]{$\hobar_0$} (c1);
\draw[->] (c2) to[bend right=60] node(f1){} node[below]{$\hubar$} (c3);
\draw[->] (c2) to[bend left=30] node(g1){} node[midway, fill=white]{$\hubar_0$} (c3);
\draw[->] (c0) to node(y){} node[left]{$w$} (c2);
\draw[->] (c1) to node(x){} node[right]{$y$} (c3);
\draw[double,double equal sign distance,-implies, shorten >=2] (e2) to node[left]{$\betubar$} (e1);
\end{tikzpicture}
\]
Then $h = \tricell{h}$ is indeed a lift of $\hubar$. Moreover, if $\alpha_0$ is a lift of $2$-cell $\alubar \maps \fubar \oo \hubar \To \gubar$ as in part (i) of Remark~\ref{rem:explicit-material-of-weak-cart}, then obviously $\alubar_0 = \alubar \oo (\fubar \dot \betubar)$, and it follows that $\alpha= \pair{\alobar}{\alubar}$ is a 2-cell in $\KK_{\CD}$ from $f\oo h$ to $g$ which lies over $\alubar$.
\end{proof}

\begin{proof}[of Proposition~\ref{pro:cart-1-cell-aka-bipullback}]
We first prove the only if part. Suppose that $f \maps y \to x$ is a cartesian $1$-cell in $\KK_{\CD}$.
For each object $c$ of $\KK$,
let us write $\mathrm{WCone}(c;x,\fubar)$
for the category of weighted cones (in the pseudo- sense)
from $c$ to the opspan $(x, \fubar)$,
in other words pairs of 1-cells
$\gobar \maps c \to \xobar$
and $\hubar\maps c \to \yubar$ as in diagram below,
and equipped with an iso-2-cell
$\gtri \maps x \oo \gobar \To \fubar \oo \hubar$.
We have chosen the notation so that if we define
$\gubar = \fubar\oo\hubar$,
and if we allow $c$ also to denote the identity on $c$
as 0-cell in $\KK_{\CD}$,
then $g \maps c \to x$ is a 1-cell in $\KK_{\CD}$.

Then for each $c$ we have a functor
$F_c \maps \KK(c,\yobar) \to \mathrm{WCone}(c;x,\fubar)$,
given by $\hobar \mapsto (\fobar\oo\hobar,y\oo\hobar)$,
with the iso-2-cell got by whiskering $\ftri$,
and we must show that each $F_c$ is an equivalence
of categories.

First we deal with essential surjectivity.
Since $f$ is cartesian we can lift $\hubar$ and the identity 2-cell $\fubar \oo \hubar = \gubar$
to a 1-cell $h\maps c \to y$ in $\KK_{\CD}$
with isomorphism
$\iota = \pair{\iotobar}{\id} \maps f \oo h \To g$,
where we have used Lemma~\ref{lem:strict-lift-replacement} to obtain $h$ as a lift rather than a weak lift.
\[
\begin{tikzpicture}[baseline=(c2.base),scale=1,every node/.style={transform shape}]
\node(c0) at (-1.7,2) {$c$};
\node(c0down) at (-1.7, 0) {$c$};
\node(c1) at (2,2) {$\xobar$};
\node(c2) at (0,0) {$\yubar$};
\node(c3) at (2,0) {$\xubar$};
\node(c4) at (0,2) {$\yobar$};
\node(c5) at (0.1,2) {};
\node(d0) at (1,1) {$\ftdar$};
\node(d1) at (-0.7,1.1) {$\htdar$};
\draw[->] (c2) to node(g1){} node[below]{$\fubar$} (c3);
\draw[-,double] (c0) to (c0down);
\draw[->] (c0down) to node[below]{$\hubar$} (c2);
\draw[->] (c0) to[bend left= 70] node(g3){} node[above]{$\gobar$} (c1);
\draw[->] (c4) to node(g4){} node[midway, fill=white]{$\fobar$} (c1);
\draw[->, dashed] (c0) to node(g5){} node[midway, fill=white]{$\hobar$} (c4);
\draw[->] (c1) to node(g6){} node[right]{$x$} (c3);
\draw[->] (c4) to node(g7){} node[right]{$y$} (c2);
\draw[double,double equal sign distance,-implies, shorten <=4, shorten >=4] (c5) to node[left]{$\iotobar$} (g3);
\end{tikzpicture}
\]

To prove that $F_c$ is full and faithful,
take any $1$-cells $\hobar$ and $\hobar'$ in $\KK$.
In the diagram above we can define $\hubar = y\oo\hobar$
and $\htri$ the identity 2-cell on $\hubar$ to get a 1-cell
$h\maps c \to y$ in $\KK_{\CD}$,
and similarly $h' \maps c \to y$.

Now suppose we have $2$-cells $\delubar \maps y \hobar \To y \hobar'$ and $\sigobar \maps \fobar \hobar \To \fobar \hobar'$ such that they form a weighted cone over $\fubar$ and $x$, \ie they satisfy compatibility equation
\[
  (\fubar \dot \delubar) \oo (\ftri \dot \hobar)
     = (\ftri \dot \hobar') \oo (x \dot \sigobar)
  \text{.}
\]
If we define $\sigubar = \fubar\dot\delubar$,
then that equation tells us that $\sigma = (\sigobar,\sigubar)$
is a 2-cell from $fh$ to $fh'$ in $\KK_{\CD}$.
Now the cartesian property tells us that there is a unique
$\delta\maps h \to h'$
over $\delubar$ such that $f\dot\delta = \sigma$,
and this gives us the unique $\delobar\maps\hobar\To\hobar'$
that we require for $F_c$ to be full and faithful.

Conversely, suppose that $\fobar$ and $y$ exhibit $\yobar$ as the bipullback of $\fubar$ and $x$ as illustrated in diagram \ref{diag:2-cod-cart}.
We show that $f\maps y \to x$ is a cartesian $1$-cell in $\KK_{\CD}$,
in other words that, for every $w$,
the functor $G_w = \left\langle P_{w,y}, f_{\ast}\right\rangle$
in Definition~\ref{defn:weakly-cart} is an equivalence.

To prove essential surjectivity,
assume that a $1$-cell $g \maps w \to x$ in $\KK_{\CD}$ is given together with a $1$-cell $\hubar\maps \wubar \to \yubar$ and an iso-2-cell $\alubar \maps \fubar \hubar \To \gubar$ in $\KK$.
\[
\begin{tikzpicture}[baseline=(c2.base),scale=1,every node/.style={transform shape}]
\node(c0) at (0,2) {$\wobar$};
\node(c1y) at (2,2) {$\yobar$};
\node(c1x) at (4,2) {$\xobar$};
\node(c2) at (0,0) {$\wubar$};
\node(c3y) at (2,0) {$\yubar$};
\node(c3x) at (4,0) {$\xubar$};
\node(d0) at (3,1) {$\ftdar$};
\node(d1) at (1,1) {$\htdar$};
\node(d5) at (2,-0.6) {$\alubar\Downarrow$};
\draw[->] (c2) to node[below]{$\hubar$} (c3y);
\draw[->] (c3y) to node[below]{$\fubar$} (c3x);
\draw[->] (c0) to node(g1){} node[left]{$w$} (c2);
\draw[->] (c1y) to node(g1){} node[above]{$\fobar$} (c1x);
\draw[->] (c1y) to node(g1){} node[left]{$y$} (c3y);
\draw[->] (c1x) to node(g1){} node[right]{$x$} (c3x);
\draw[->] (c0) to[bend left= 60] node(g3){} node[above]{$\gobar$} (c1x);
\draw[->] (c2) to[bend right= 60] node(g3){} node[below]{$\gubar$} (c3x);
\draw[->, dashed] (c0) to node[midway, fill=white]{$\hobar$} (c1y);
\end{tikzpicture}
\]
The iso-2-cell
$\gamma\eqdef (\alubar^{-1} \dot w ) \oo \gtri \maps x\gobar \To \gubar w \To \fubar\hubar w$
factors through the bipullback 2-cell with apex $\yobar$, and therefore it yields a $1$-cell $\hobar \maps \wobar \to \yobar$
and
iso-2-cells $\htri \maps y \oo \hobar \To \hubar \oo w$
(making a 1-cell $h\maps w \to y$ in $\KK_{\CD}$)
and $\alobar\maps  \fobar \oo \hobar \To \gobar$ such that $\ftri$ and $\htri$ paste to give
$\gamma \oo (x \dot \alobar)$.
From this we observe that $h\eqdef\tricell{h}$ is a lift of $\hubar$ and
$\alpha\eqdef\pair{\alobar}{\alubar}$
is an iso-2-cell from $fh$ to $g$ over $\alubar$
as required for cartesianness.

To show that $G_w$ is full and faithful,
suppose we have 1-cells $h,h' \maps w \to y$.
If $\delubar\maps\hubar\To\hubar'$
and $\sigma\maps fh \To fh'$
with $\fubar\dot\delubar = \sigubar$,
we must show that there is a unique
$\delta\maps h \To h'$ over $\delubar$
with $f\dot\delta = \sigma$.

We have 2-cells $\sigobar\maps \fobar \hobar \To \fobar \hobar'$
\[
  \mu = (\htri^{\prime -1})(\delubar\dot w)(\htri)
    \maps y\hobar \To \hubar w \To \hubar' w
      \To y\hobar
   \text{,}
\]
and moreover
\[
  (\ftri \dot\hobar')(x\dot\sigobar)
  = (\fubar \dot \htri^{\prime -1})(f h')^{\blacktriangledown}(x\dot\sigobar)
  = (\fubar \dot \htri^{\prime -1})(\sigubar\dot x)(f h)^{\blacktriangledown}
  = (\fubar\dot\mu)(\ftri\dot\hobar)
  \text{.}
\]
It then follows from the bipullback property that
we have a unique $\delobar\maps\hobar\To\hobar'$
such that $y\dot\delobar = \delubar\dot w$
(so we have a 2-cell $\delta\maps h \To h'$
over $\delubar$)
and $\fobar\dot\delobar = \sigobar$,
so $f\dot\delta = \sigma$ as required.

\end{proof}

%%%%%%%%%%%%%%%%%%%%%%%%%%%%%%%%%%%%%%%%%%%%
\section{Johnstone's criterion} \label{sec:Johnstone-crit}
%%%%%%%%%%%%%%%%%%%%%%%%%%%%%%%%%%%%%%%%%%%%
Another definition of (op)fibration first appeared in~\cite{johnstone:fib-par-prod};
see also~\cite[B4.4.1]{johnstone:elephant1} for more discussion.
This definition does not require strictness of the $2$-category nor the existence of comma objects.
Indeed, this definition is most suitable for weak $2$-categories such as
$2$-category of toposes where we do not expect diagrams of $1$-cells to commute strictly.
Moreover, although Johnstone assumed the existence of bipullbacks,
in fact one only needs bipullbacks of the class of $1$-cells one would like to define as (op)fibrations.
This enables us to generalize some of Johnstone's results from $\BTopos$
(where all bipullbacks exist)
to $\ETopos$ (where bounded 1-cells are bicarrable).

We have adjusted axiom (i) (lift of identity) in Johnstone's definition so that the (op)fibrations we get have the right weak properties.
That is to say, unlike Johnstone's definition, we only require lift of identity to be isomorphic (rather than equal) to identity.

The Johnstone definition is rather complicated, as it has to deal with coherence issues.
We have found a somewhat simpler formulation,
so we shall first look at that.
It is simpler notationally, in that it uses single symbols to describe two levels
of structure, ``downstairs'' and ``upstairs''.
More significantly, it is also simpler structurally in that it doesn't assume
canonical bipullbacks and then describe the coherences between them.
Instead it borrows techniques from the 2- and bi-categorical theories of fibrations,
which use the existence of cartesian liftings as bipullbacks.
This enables us to show
(Proposition~\ref{pro:Johnstone-fib-as-cod-fibrant-objects})
that the Johnstone criterion is equivalent to the
fibrational property of Definition~\ref{def:Johnstone-fib-via-Buckley}.

\begin{defn}
\label{defn:Johnstone-2-fib}
Suppose $\KK$ is a 2-category.
A $1$-cell $x\maps \xobar \to \xubar$ in $\KK$ \strong{satisfies the Johnstone criterion} if the following two conditions hold.
First, $x$ is bicarrable.
Second, suppose $\fubar,\gubar \maps \yubar \toto \xubar$
are two 1-cells into $\xubar$,
with $x_f\maps \xobar_f \to \xubar_f = \yubar$
and $x_g\maps \xobar_g \to \xubar_g = \yubar$
two provided bipullbacks of $x$ along $\fubar$ and $\gubar$.

Then for any $2$-cell $\alubar \maps \fubar \To \gubar$,
we have a $1$-cell $\raltop \maps \xobar_g \to \xobar_f$,
an iso-$2$-cell $\ralmid \maps x_f  \oo  \raltop \To x_g$,
and a $2$-cell $\alobar \maps \fobar \oo \raltop \To \gobar$
shown in the diagram on the left below.
(The canonical iso-$2$-cells $\ftri \maps x \oo \fobar \To \fubar \oo x_f$
and $\gtri \maps x \oo \gobar \To \gubar \oo x_g$
for the bipullbacks front and back are not shown.)
For convenience we show on the right the same diagram but with the notation
of~\cite[Definition B4.4.1]{johnstone:elephant1}.
\begin{equation}
%*****************diagram
\begin{tikzcd}[column sep=3.5em, row sep=1.5em]
& {\xobar_g}
\arrow[dl, swap, "\raltop"]
\arrow[rr,"\gobar"] \arrow[rd, phantom,bend left=35, ""{name=qux}] \arrow[dd, near start, swap, "x_g" description]
&& {\xobar}
\arrow[dl,equal]
\arrow[dd,"x"] \\
\xobar_f
\arrow[rr, crossing over, pos=0.65, swap, "\fobar"]
\arrow[dd, swap, "x_f"]
\arrow [rr, phantom, bend left=10, ""{name=baz}] \arrow[rd, phantom, bend right=15,""{name=fred}]
\arrow[rd, phantom, bend left=15, ""{name=waldo}]
&& \xobar
\arrow[dd, pos=0.35, "x"]
\arrow[Rightarrow, from=baz, to=qux,shorten <= +23pt, shorten >= +1pt, shift right=10pt, pos=0.8, "\alobar"]
\arrow[Rightarrow, phantom, from=fred, to=waldo, swap, "\raldnar"] \\
& {\yubar}
\arrow[dl, equal] \arrow[rr, pos=0.35, "\gubar"]
\arrow[rd, phantom, bend left=35, ""{name=bar}] &&
{\xubar}
\arrow[dl, equal]  \\
{\yubar}
\arrow[rr, swap, "\fubar"]
\arrow[rr, phantom,""{name=foo}, bend left=10]
&& {\xubar}
\arrow[from=uu, crossing over]
\arrow[Rightarrow, from=foo, to=bar, shorten <= +20pt, shorten >= +5pt, shift right=10pt, pos=0.6, "\alubar"] \\
\end{tikzcd}
%%%%%%%%%%% end of diagram
\quad
%%%%%%%%%%% Elephant diagram
\begin{tikzcd}[column sep=3.5em, row sep=1.5em]
& {g^\ast E}
  \arrow[dl, swap, "l(\alpha)"]
  \arrow[rr,"p^\ast g"]
  \arrow[rd, phantom,bend left=35, ""{name=qux}]
  \arrow[dd, near start, swap, "g^\ast p" description]
&& {E}
  \arrow[dl,equal]
  \arrow[dd,"p"]
\\
{f^\ast E}
  \arrow[rr, crossing over, pos=0.65, swap, "p^\ast f"]
  \arrow[dd, swap, "f^\ast p"]
  \arrow [rr, phantom, bend left=10, ""{name=baz}]
  \arrow[rd, phantom, bend right=15,""{name=fred}]
  \arrow[rd, phantom, bend left=15, ""{name=waldo}]
&& {E}
  \arrow[dd, pos=0.35, "p"]
  \arrow[Rightarrow, from=baz, to=qux,shorten <= +23pt, shorten >= +1pt,
    shift right=10pt, pos=0.8, "\tilde{\alpha}"]
  \arrow[Rightarrow, phantom, from=fred, to=waldo, swap, "\iso"]
\\
& {A}
  \arrow[dl, equal]
  \arrow[rr, pos=0.35, "g"]
  \arrow[rd, phantom, bend left=35, ""{name=bar}]
&& {B}
  \arrow[dl, equal]
\\
{A}
  \arrow[rr, swap, "f"]
  \arrow[rr, phantom,""{name=foo}, bend left=10]
&& {B}
  \arrow[from=uu, crossing over]
  \arrow[Rightarrow, from=foo, to=bar, shorten <= +20pt,
    shorten >= +5pt, shift right=10pt, pos=0.6, "\alpha"]
\\
\end{tikzcd}
\end{equation}
We simplify this by taking $\CD$ to be the class of all bicarrable 1-cells
in $\KK$ and working in $\KK_{\CD}$.
(We could equally well work with $\CD$ any class of display 1-cells in $\KK$,
as in Construction~\ref{cons:display-sub-2-cat}.)
Thus we have cartesian 1-cells $f\maps x_f \to x$
and $g\maps x_g \to x$,
and a \emph{vertical} 1-cell $r_\alpha \maps x_g \to x_f$
($\xubar_g = \yubar = \xubar_f$, and $\rubar_\alpha$ is the identity).

The data is subject to the following axioms:
\begin{enumerate}[label=(\subscript{J}{{\arabic*}})]
%***************** first axiom
\item\label{first-axiom-fib}
  $\alpha = (\alobar, \alubar)$ make a 2-cell in $\KK_\CD$,
  so we get the following diagram.
  \begin{equation}
    \begin{tikzcd}[column sep=3.5em, row sep=1.5em]
      & {x_g}
      \arrow[dl, swap, "r_\alpha"]
      \arrow[rr,"g"]
      \arrow[rd, phantom,bend left=35, ""{name=qux}]
      && {x}
      \arrow[dl,equal]
      \\
      {x_f}
      \arrow[rr, pos=0.65, swap, "f"]
      \arrow [rr, phantom, bend left=10, ""{name=baz}]
      && {x}
      \arrow[Rightarrow, from=baz, to=qux,
        shorten <= +23pt, shorten >= +1pt,
        shift right=10pt, pos=0.8, "\alpha"]
    \end{tikzcd}
  \end{equation}
\item\label{second-axiom-fib}
  Suppose we have two composable $2$-cells $\alubar\maps \fubar \To \gubar$ and
  $\betubar\maps \gubar \To \hubar$ in $\KK$ where
  $\fubar, \gubar, \hubar \maps \yubar \to \xubar$;
  we write $\gamubar\eqdef\betubar\alubar$.
  Let $\alpha,\beta,\gamma,r_\alpha,r_\beta,r_\gamma$
  be as above.
  Then there exists a vertical iso-$2$-cell
  $\tau_{\alpha, \beta} \maps r_{\alpha} \oo r_{\beta} \To r_{\gamma}$ in $\KK_{\CD}$ such that
  $\beta \oo (\alpha \dot r_{\beta}) \oo (f \dot \tau^{-1}_{\alpha, \beta})=\gamma$.
  \[
  \begin{tikzcd}[row sep=3em]
    x_h
      \arrow{d}[swap]{r_{\beta}}
      \arrow[dd, bend right= 60, swap, "r_{\gamma}", ""{name=DD}]
      \arrow[dd, phantom, bend left= 60, ""{name=DD0}]
      \arrow[ddr, bend left= 30, "h", ""{name=DDR}]
    \\
    x_g
      \arrow{d}[swap]{r_{\alpha}}
      \arrow[""{name=bar, below},dr, "g" description]
      \arrow[dr, phantom, ""{name=foo}, bend right=55]
    \\
    x_f
      \arrow{r}[swap]{f}
      \arrow[Rightarrow, from= foo, to=bar, swap, near start,
        "\alpha", shorten >=3]
    & x
    \arrow[Rightarrow, from=bar, to=DDR, swap, pos=0.7, "\beta",
      shorten <=15, shorten >=4, shift left =5]
    \arrow[Rightarrow, from= DD0, to=DD, swap, "\tau_{\alpha, \beta}",
      pos= 0.8, shorten <=45, shorten >=6]
  \end{tikzcd}
  \]
\item\label{third-axiom-fib}
  For any $1$-cell $\fubar \maps \yubar \to \xubar$ the lift of
  the identity $2$-cell on $\fubar$ is canonically isomorphic
  to the identity $2$-cell on the lift $f$,
  that is there exists a vertical iso-$2$-cell
  $\tau_{f} \maps 1_f \To r_{id_{\fubar}}$ in $\KK_{\CD}$ such that
  $f \dot \tau^{-1}_{f}$ is the lift of identity $2$-cell $id_{\fubar}$.
  \[
  \begin{tikzcd}[row sep=3em]
    x_f
      \arrow[d, swap, "1" description]
      \arrow[""{name=bar, below},dr, "f"]
      \arrow[""{name=downbent, below},d, bend right = 60, swap, "r_{\id}"]
      \arrow[""{name=downbent0, below}, bend left = 60, phantom, d]
      \arrow[dr, phantom, ""{name=foo}, bend right=55]
    \\
    x_f
      \arrow{r}[swap]{f}
      \arrow[Rightarrow, equal, from= foo, to=bar, swap, near start,
        shorten >=3]
    & x
      \arrow[Rightarrow, "\tau_f", from= downbent0, to=downbent, swap,
        near start, shorten <=28, shorten >=1]
  \end{tikzcd}
  \]
\item\label{fourth-axiom-fib}
  The lift of the whiskering of any $2$-cell
  $\alubar \maps \fubar \to \gubar \maps \yubar \to \xubar$
  with any
  $1$-cell $k \maps \zubar \to \yubar$ is isomorphic,
  via vertical iso-$2$-cells,
  to the whiskering of the lifts.

  In the following diagram,
  the right hand square is as usual,
  $f'$ and $g'$ are cartesian lifts of $\fubar\kubar$
  and $\gubar\kubar$,
  and the 1-cells $k_f$ and $k_g$ over $\kubar$
  and the vertical iso-2-cells $\rho$ and $\pi$
  are got from cartesianness of $f$ and $g$.
  Then the condition is that there should be a vertical
  iso-2-cell in the left hand square,
  which pastes with the others to give
  $\alpha'\maps f' r_{\alpha'} \To g'$
  of $\alubar \dot \kubar$.
  \begin{equation*}
  \begin{tikzcd}[column sep =3.1em,row sep =3.1em]
    x_{g'}
      \arrow[d, swap, ""{name=WNS}, "r_{\alpha'}"]
      \arrow[r, "k_g"]
      \arrow[rr, bend left=45, "g'" ,""{name=bend up2}]
      \arrow[rr, bend left=5, phantom, ""{name=bend up1}]
    & x_g
      \arrow[d, "r_{\alpha}", ""{name=MNS}]
      \arrow[r, "g", ""{name=EN}]
    & x
      \arrow[d,equal]
    \\
    x_{f'}
      \arrow[r, swap, "k_f"]
      \arrow[rr, bend right=45, "f'",swap, ""{name=bend down1}]
      \arrow[rr, bend right=5, phantom, ""{name=bend down2}]
    & x_f
      \arrow[r,swap, "f", ""{name=ES}]
    & x
    \arrow[phantom, Rightarrow, from= WNS, to=MNS, "\iso"]
    \arrow[phantom, Rightarrow, from= bend up1, to=bend up2,
      "\iso\ \pi"]
    \arrow[phantom,Rightarrow,from=bend down1,to=bend down2,
      "\iso\ \rho"]
    \arrow[Leftarrow, from= EN, to=ES, shorten <=20, shorten >=20,
      "\alpha"]
  \end{tikzcd}
  \end{equation*}
\item\label{fifth-axiom-fib}
  Given any pair of vertical $1$-cells $v_0 \maps y \to x_f$
  and $v_1 \maps y \to x_g$,
  any $2$-cell $\alpha_0 \maps f \oo v_0 \To g \oo v_1$
  over $\alubar$ factors through $\alpha$ uniquely,
  that is there exists a unique vertical $2$-cell
  $\mu \maps v_0 \To r_{\alpha} v_1$
  such that the following pasting diagrams are equal.
  \[
  \begin{tikzcd}
      [column sep =4.1em,row sep =4.1em, ampersand replacement=\&]
    y
      \arrow[d,swap,"v_1"]
      \arrow[r,phantom, bend left =20, ""{name=NWE}]
      \arrow[r,"v_0"]
    \& x_f \arrow[d, "f"]
    \\
    x_g
      \arrow[r,"g",swap]
      \arrow[r,phantom, bend right =20, ""{name=SWE}]
    \& x
    \arrow[Rightarrow, from=NWE, to=SWE, pos=0.5 ,shorten <=30,
      shorten >=30,"\alpha_0"]
    \end{tikzcd}
    \quad=\quad
    \begin{tikzcd}
        [column sep =4.1em,row sep =4.1em, ampersand replacement=\&]
      y
        \arrow{d}[swap]{v_1}
        \arrow[r,"v_0"]
        \arrow[r,bend left=70, phantom, ""{name=NWE1}]
        \arrow[r,bend right=50, phantom, ""{name=NWE2}]
      \& x_f
        \arrow[d, "f"]
      \\
      x_g
        \arrow[r,swap,"g"]
        \arrow[r,bend right=50, phantom, ""{name=SWE2}]
        \arrow[r,bend left=70, phantom, ""{name=SWE1}]
      \arrow[""{name=diag}, pos=0.5, ur,"r_{\alpha}" description] \& x
      \arrow[Rightarrow, swap, from=NWE1, to=NWE2, shorten <=28,
        shorten >=0.01, shift right=2, pos=0.8, "\mu"]
      \arrow[Rightarrow, from=SWE1, to=SWE2, shorten <=4,
        shorten >=24,pos=0.2, "\alpha"]
  \end{tikzcd}
  \]
\end{enumerate}
\end{defn}

\begin{rem}
Dually, \strong{opfibrations} are defined by changing the direction of $r_\alpha$:
for each $\alubar\maps\fubar\to\gubar$
we require a $1$-cell $l_\alpha \maps x_f \to x_g$
and $\alpha\maps f \To gl_\alpha$ with the axioms modified accordingly.
The letters $r$ and $l$ used here correspond to
Street's $2$-monads $R$ and $L$
in~\cite{street:fib-yoneda-2cat}.
\end{rem}

\begin{prop}
A fibration $p\maps E \to B$ is also an opfibration precisely when every $2$-cell $\alubar$ induces an adjunction $\elaltop \adj \raltop$.
\end{prop}

\begin{proof}
The unit and counit of adjunction are obtained by choosing $(1_{\strund{f}E}, \elaltop)$ and $(\raltop, 1_{{\strund{g}E}})$ for $(\xobar,\yobar)$ in axiom \ref{fifth-axiom-fib} above.
\end{proof}

% \begin{rem}
% The last axiom states that the fibration $x$ has enough precartesian lifts.%
% Conditions~\ref{third-axiom-fib} and~\ref{fourth-axiom-fib} then state that precartesian lifts are (up to isomorphism) closed under composition. Recall that a (cloven) prefibration is a (cloven) fibration if and only if precartesian morphisms are closed under composition. Hence this definition is reminiscent of characterization of Grothendieck fibration in terms of existence of precartesian lifts.
% \end{rem}

\begin{ex}
Let's take $\Cat$ to be the $2$-category of (small) categories, functors and natural transformations. Here we show that an internal fibration in $\Cat$ is indeed something that is referred to as a \textit{weak fibration} in the literature, e.g. in~\cite{street:consp}.

A functor of categories $p\maps E \to B$ is a Grothendieck fibration if and only if for every object $e$ of $E$,
the slice functor $p/e\maps E/e \to B/p(e)$ has a right adjoint right inverse.
It is a weak fibration whenever it has a right adjoint. Weak fibrations are also known by the names Street fibrations and sometimes \textit{abstract} fibrations. One can associate to every weak fibration an equivalent Grothendieck fibration. That is, every Street fibration can be factored as an equivalence followed by a Grothendieck fibration.

Let $p\maps E \to B$ be a Johnstone fibration in $\Cat$. Let $1$ be the terminal category, $e \in E$ and $\alubar\maps b \to pe$ a morphism in $B$.
\[
\begin{tikzcd}[row sep=3em]
1 \arrow{r}{e} \arrow[""{name=foo}, dr,swap, bend right=15, "b"] \arrow[dr, phantom, ""{name=bar}, bend left=55]  \arrow[Rightarrow, swap, from= foo, to= bar, shorten <=10, pos=0.6, "\alubar"]& E \arrow{d}{p} \\
& B
\end{tikzcd}
\]

$b^*E$ has as objects all pairs $\lang x \in E, \sigma \maps p x \iso b \rang$, and as morphisms all morphisms $h \maps x \to x'$ in $E$ making the triangle
\[
\begin{tikzcd}[column sep=small]
px \arrow[dr, swap, "\sigma"] \arrow{rr}{ph} & & px' \arrow[dl, "\sigma'"] \\
& b & \\
\end{tikzcd}
\]
commute.
Similarly, the bipullback category $(pe)^*E$ can be described. Notice that $\lang e, id_{pe} \rang$ is an object of $(pe)^*E$. Applying $\raltop$ we get an object $x$ in $E$ with an isomorphism $\sigma\maps px \iso b$. Axiom \ref{first-axiom-fib} implies $p(\alobar)= \alubar \oo \sigma$. $\alobar$ is the lift of $\alpha$ and axioms \ref{fourth-axiom-fib} and \ref{fifth-axiom-fib} prove that this lift is cartesian. Axioms \ref{second-axiom-fib} and \ref{third-axiom-fib} give coherence equations of lifts for identity and composition.
\end{ex}

Our goal now (Proposition~\ref{pro:Johnstone-fib-as-cod-fibrant-objects}) is to show that,
for the $2$-functor $\ccod\maps \KK_{\CD} \to \KK$,
a 1-cell $x\maps \xobar \to \xubar$ in $\KK$ satisfies the Johnstone criterion
iff it is fibrational for $\ccod$ in the sense of
Definition~\ref{def:Johnstone-fib-via-Buckley}.

\begin{lem}
\label{lemma:cartesianness-of-Johnstone-lift-2-cell}
  Suppose $x$ in $\KK_{\CD}$ satisfies the
  Johnstone criterion of Definition~\ref{defn:Johnstone-2-fib}.
  Let $\fubar$, $\gubar$ and $\alubar$ be
  as in the definition,
  giving rise to $f\maps x_f \to x$, $g\maps x_g\to x$
  and $\alpha\maps f r_\alpha \To g$,
  and let $u\maps z \to x_g$ be any 1-cell in $\KK_\CD$.
  Then the whiskering
  $\alpha\dot u\maps f r_\alpha u \To gu$ is cartesian.
\end{lem}
\begin{proof}
First, we deal with the case where $u$ is vertical.
Note that this also shows that $\alpha$ itself
is cartesian.

Suppose $\gamma_0 \maps e_0 \To gu$ is a 2-cell in $\KK_{\CD}$ such that
$\ccod(\gamma_0) = \gamubar_0 = \alubar \oo \betubar$ in $\KK$.
We seek a unique 2-cell $\beta_0 \maps e_0 \To f r_{\alpha} u$
over $\betubar$ such that $(\alpha \dot u) \oo \beta_0 = \gamma_0$.

Let $e \maps  x_e \to x$ be a cartesian lift of $\eubar_0$,
obtained as a bipullback.
Then we can factor $e_0$, up to a vertical iso-2-cell,
as $ev$ where $v$ is a vertical $1$-cell.
We can neglect the iso-2-cell and assume $e_0=ev$.
Also, let $\beta \maps e \oo r_{\beta} \To f$
and $\gamma \maps e \oo r_{\gamma} \To g$
be lifts of $\betubar \maps \eubar = \eubar_0 \To \fubar$
and $\gamubar \eqdef \gamubar_0\maps \eubar \To \gubar$
obtained from the fibration structure of $x$.

From axiom~\ref{second-axiom-fib} we get an iso-2-cell
$\tau_{\beta,\alpha}\maps r_\beta \oo r_\alpha \to r_\gamma$.
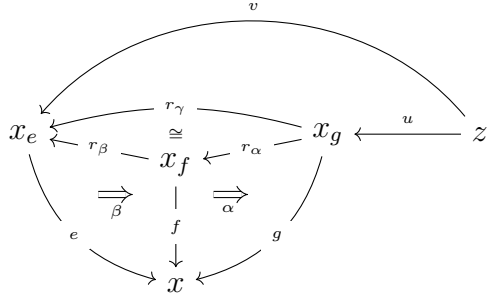
\[
\begin{tikzpicture}
\node(c1) at (-2,0) {$x_e$};
\node(c2) at (0,-0.4) {$x_f$};
\node(c3) at (2,0) {$x_g$};
\node(c5) at (4,0) {$z$};
\node(c4) at (0,-2) {$x$};
\node(c7) at (-1.2,-0.8) {};
\node(c8) at (0,-0.8) {};
\node(c9) at (1.2,-0.8) {};
\node(c10) at (0,0) {\tiny $\iso$};
\draw[->] (c2) to node[midway, fill=white]{\tiny $r_\beta$} node(f1){} (c1);
\draw[->] (c3) to node[midway, fill=white]{\tiny $r_\alpha$} node(f2){} (c2);
\draw[->] (c1) to[bend right=30] node[midway, fill=white]{\tiny $e$} node(e){} (c4);
\draw[->] (c2) to node[midway, fill=white]{\tiny $f$} node(f){} (c4);
\draw[->] (c3) to[bend left=30] node[midway, fill=white]{\tiny $g$} node(g){} (c4);
\draw[->] (c3) to[bend right=15] node[midway, fill=white]{\tiny $r_\gamma$} node(f6){} (c1);
\draw[->] (c5) to node[above]{\tiny $u$} (c3);
\draw[->] (c5) to[bend right=50] node[above]{\tiny $v$} node(f5){} (c1);
\draw[double,double equal sign distance,-implies, shorten <=1, shorten >=12] (c7) -- node[below, pos=0.3]{\tiny $\beta$} (c8);
\draw[double,double equal sign distance,-implies, shorten <=10, shorten >=3] (c8) -- node[below, pos=0.6]{\tiny $\alpha$} (c9);
\end{tikzpicture}
\]

Using axiom~\ref{fifth-axiom-fib},
the unique $\beta_0\maps ev \To f r_\alpha u$ that we seek
amounts to a unique vertical
$\mu_0\maps v \To r_\beta \oo r_\alpha \oo u$
such that the diagram on the left below pastes with $\alpha \dot u$
to give $\gamma_0\maps ev \To gu$.
Bringing in $\tau{_\beta,\alpha}$,
this amounts to finding a unique vertical
$\mu_1\maps v \To r_\gamma \oo u$
such that the equation on the right holds,
and this is immediate from axiom~\ref{fifth-axiom-fib}.

\begin{equation*}
\begin{tikzcd}[column sep =4.1em,row sep =4.1em]
  z
    \arrow{d}[swap]{r_\alpha u}
    \arrow[r,"v"]
    \arrow[r,bend left=70, phantom, ""{name=NWE1}]
    \arrow[r,bend right=70, phantom, ""{name=NWE2}]
  & x_e
    \arrow[d, "e"]
  \\
  x_f
    \arrow[r,swap,"f"]
    \arrow[r,bend right=50, phantom, ""{name=SWE2}]
    \arrow[r,bend left=70, phantom, ""{name=SWE1}]
    \arrow[""{name=diag}, pos=0.5, ur,"r_{\beta}" description]
  & x
  \arrow[Rightarrow, from=NWE1, to=NWE2, shorten <=25,shorten >=0.11,
    shift right=2, pos=0.8, "\mu_0"]
  \arrow[Rightarrow, from=SWE1, to=SWE2, shorten <=1, shorten >=20,
    pos=0.2, "\beta"]
\end{tikzcd}
=
\begin{tikzcd}[column sep =4.1em,row sep =4.1em]
  z
    \arrow{d}[swap]{r_\alpha u}
    \arrow[r,"v", ""{name=NWE}]
  & x_e \arrow[d, "e"]
  \\
  x_f
    \arrow[r,"f",swap, ""{name=SWE}]
  & x
  \arrow[Rightarrow, from=NWE, to=SWE, pos=0.5 ,shorten <=20,
    shorten >=20,"\beta_0"]
\end{tikzcd}
\quad\quad
\begin{tikzcd}[column sep =4.1em,row sep =4.1em]
  z
    \arrow{d}[swap]{u}
    \arrow[r,"v"]
    \arrow[r,bend left=70, phantom, ""{name=NWE1}]
    \arrow[r,bend right=70, phantom, ""{name=NWE2}]
  & x_e
    \arrow[d, "e"]
  \\
  x_g
    \arrow[r,swap,"g"]
    \arrow[r,bend right=50, phantom, ""{name=SWE2}]
    \arrow[r,bend left=70, phantom, ""{name=SWE1}]
    \arrow[""{name=diag}, pos=0.5, ur,"r_{\gamma}" description]
  & x
  \arrow[Rightarrow, from=NWE1, to=NWE2, shorten <=25,shorten >=0.11,
    shift right=2, pos=0.8, "\mu_1"]
  \arrow[Rightarrow, from=SWE1, to=SWE2, shorten <=1, shorten >=20,
    pos=0.2, "\gamma"]
\end{tikzcd}
=
\begin{tikzcd}[column sep =4.1em,row sep =4.1em]
  z
    \arrow[d,swap,"u"]
    \arrow[r,"v",""{name=NWE}]
  & x_e \arrow[d, "e"]
  \\
  x_g \arrow[r,"g",swap, ""{name=SWE}]
  & x
  \arrow[Rightarrow, from=NWE, to=SWE, pos=0.5 ,shorten <=20,
    shorten >=20,"\gamma_0"]
\end{tikzcd}
\end{equation*}

Now we prove the result for general $u$.

We can factor $u$ up to an iso $2$-cell as $kv$,
where $v$ is vertical and $k$ is cartesian.
Because of Lemma~\ref{lem:closure-properties-of-cart-1,2-cells} \ref{lem:item:closed-under-composition},\ref{lem:item:iso-2-cell}
we might as well assume that $u = kv$.
Axiom \ref{fourth-axiom-fib} implies that, up to an iso-2-cell, $\alpha \dot k$ can be obtained as the lift of $\alubar \dot \kubar$.
We can thus apply the vertical case, already proved,
to see that $(\alpha \dot k)\dot v$ is cartesian.
\end{proof}

\begin{prop}
\label{pro:Johnstone-fib-as-cod-fibrant-objects}
A $1$-cell $x\maps \xobar \to \xubar$ in $\KK$ is a fibration in the sense of the Johnstone criterion Definition~\ref{defn:Johnstone-2-fib} iff
it is fibrational as a 0-cell in $\KK_{\CD}$
(Definition~\ref{def:Johnstone-fib-via-Buckley}).
\end{prop}

\begin{proof}
First note, by Proposition~\ref{pro:cart-1-cell-aka-bipullback},
that \ref{def:JfvB:cart-lift-1-cell} is equivalent to bicarrability of $x$.
Now suppose $x$ is a fibration in the sense of Definition~\ref{defn:Johnstone-2-fib}.

To show \ref{def:JfvB:cart-lift-2-cell},
assume that $g_0\maps y \to x$ and
$\alubar \maps \fubar \To \gubar_0 \maps \yubar \toto \xubar$
is a 2-cell in $\KK$.
We aim to find a cartesian lift of $\alubar$.

Let $f\maps x_f\to x$ and $g\maps x_g \to x$
be cartesian lifts of $\fubar$ and $\gubar_0$, so $\gubar = \gubar_0$,
and suppose the Johnstone criterion gives them structure
$\alpha\maps f r_\alpha \To g$.
Then we factor $g_0$ through $g$ and obtain a lift $v$ of $1_{\yubar}$
and an iso-$2$-cell $\mu \maps g v \To g_0$ in $\KK_{\CD}$.
Pasting $\mu$ and $\alpha$ together we get a $2$-cell
$\alpha_0 \eqdef \gamma \oo (\alpha \dot v)$,
lying over $\alubar$, from $f_0\eqdef fr_\alpha v$ to $g_0$ in $\KK_{\CD}$.
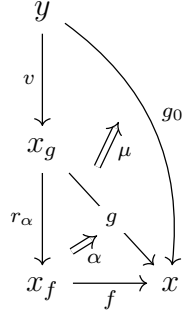
\[
  \begin{tikzcd}[row sep=3em]
    y
      \arrow{d}[swap]{v}
      \arrow[ddr, bend left= 30, "g_0", ""{name=DDR}] &
    \\
    x_g
      \arrow{d}[swap]{r_{\alpha}}
      \arrow[""{name=bar, below},dr, "g" description]
      \arrow[dr, phantom, ""{name=foo}, bend right=55]
    \\
    x_f
      \arrow{r}[swap]{f}
    & x
    \arrow[Rightarrow, from= foo, to=bar, swap, near start,
      "\alpha", shorten >=3]
    \arrow[Rightarrow, from=bar, to=DDR, swap, pos=0.7, "\mu",
      shorten <=15, shorten >=4, shift left =5]
  \end{tikzcd}
\]

Note that $\alpha_0$ is indeed cartesian. This is because $\mu$ is a an iso-2-cell,
and therefore it is cartesian by
Lemma~\ref{lem:closure-properties-of-cart-1,2-cells}\ref{lem:item:iso-2-cell},  $\alpha \dot v$ is cartesian according to Lemma~\ref{lemma:cartesianness-of-Johnstone-lift-2-cell},
and also vertical composition of cartesian 2-cells is cartesian.

For \ref{def:JfvB:whiskering-cart-2-cell},
let $\alpha_0 \maps f_0 \To g_0 \maps y \to x$ be any cartesian $2$-cell in $\KK_{\CD}$,
and let $k \maps z \to y $ any $1$-cell in $\KK_{\CD}$.
We will show that the whiskered $2$-cell $\alpha_0 \dot k$ is again cartesian.
First, let $f \maps x_f \to x$ and $g \maps x_g \to x$
be cartesian lifts of $\fubar_0$ and $\gubar_0$,
and let $\alpha\maps f r_\alpha \To g$ be got
from $\alubar_0$ in the usual way.
Then we factor $f_0$ and $g_0$ up to vertical
iso-2-cells as
$\rho \maps f_0 \iso f \oo u$
and $\pi \maps g_0 \iso g \oo v$,
where $u$, $v$ are vertical.
Define $\alpha'_0= \pi \oo \alpha_0 \oo \rho^{-1}$.
Obviously, $\alpha'_0$ is cartesian and $\alpha_0 \dot k$ is cartesian if and only if $\alpha'_0 \dot k$ is cartesian.
By axiom \ref{fifth-axiom-fib} of fibration, we get a (unique) vertical $2$-cell $\mu$ such that
$(\alpha \dot v) \oo  (f \dot \mu) = \alpha'_0$.
By Lemma~\ref{lemma:cartesianness-of-Johnstone-lift-2-cell}
$\alpha \dot v$ is cartesian and it follows that $f \dot \mu$ is cartesian since $\alpha'_0$ is cartesian.
Now the 2-cell $f \dot \mu$ is both vertical and cartesian and thus it is an iso-$2$-cell, according to Lemma~\ref{lem:closure-properties-of-cart-1,2-cells}\ref{lem:item:vert-cart-2-cell}. So, our task reduces to proving that $(\alpha \dot v) \dot k$ is a cartesian 2-cell,
and this we know from Lemma~\ref{lemma:cartesianness-of-Johnstone-lift-2-cell}.

Conversely, suppose $x \maps \xobar \to \xubar$ is a fibrational
$0$-cell in $\KK_{\CD}$.
We want to extract the structure of the Johnstone criterion for $x$ out of this data.
First of all according to \ref{def:JfvB:cart-lift-1-cell},
$x$ is bicarrable.
Suppose $\alubar \maps \fubar \To \gubar$ is any $2$-cell in $\KK$.
Let $g$ be a cartesian lift of $\gubar$ obtained as a bipullback of $\gubar$ along $x$ in $\KK$.
By \ref{def:JfvB:cart-lift-2-cell} $\alubar$ has a cartesian lift $\alpha' \maps f' \To g$.
Factor $f'$, up to an iso-$2$-cell $\gamma$, as $f \oo r_{\alpha}$ where $r_{\alpha}$ is vertical and $f\maps x_f\to x$.
From $\alpha'$ and $\gamma$ we obtain a cartesian $2$-cell $\alpha \maps f \oo r_{\alpha} \To g$ which satisfies axiom \ref{first-axiom-fib}.

\begin{equation}
\label{diag:from-Buck-to-PTJ-fib}
\begin{tikzpicture}
\node(c1) at (-1.5,0) {$x_g$};
\node(c2) at (0,-1.5) {$x_f$};
\node(c3) at (1.5,0) {$x$};
\node(c4) at (0,-1) {$\iso \gamma$};
\draw[->] (c1) to[bend right=15] node[below left]{$r_{\alpha}$} (c2);
\draw[->] (c2) to[bend right=15] node[below right]{$f$} (c3);
\draw[->] (c1) to[bend right=25] node[midway, fill=white]{$f'$} node(f2){} (c3);
\draw[->] (c1) to[bend left=35] node[above]{$g$} node(f3){} (c3);
\draw[double,double equal sign distance,-implies, shorten <=5, shorten >=4] (f2) -- node[right]{$\alpha'$} (f3);
\end{tikzpicture}
\end{equation}

To show \ref{second-axiom-fib}, take a pair of composable $2$-cells $\alubar \maps \fubar \To \gubar$ and $\betubar \maps \gubar \To \hubar$.
Carrying out the same procedure as we did in diagram \ref{diag:from-Buck-to-PTJ-fib},
we obtain cartesian 2-cells
$\alpha \maps f \oo r_{\alpha} \To g$
and $\beta \maps g \oo r_{\beta} \To h$,
and also $\gamma \maps f \oo r_{\gamma} \To h$
lifting $\gamubar = \betubar\alubar$.
By \ref{def:JfvB:whiskering-cart-2-cell},
the $2$-cell
$\beta \oo (\alpha \dot r_{\beta}) \maps
  f r_\alpha r_\beta \To h$
is cartesian.
Therefore, there exists a unique vertical iso-$2$-cell
$\sigma \maps f r_{\alpha} r_{\beta} \To f r_{\gamma}$
such that
$\gamma \oo \sigma = \beta \oo (\alpha \dot r_{\beta})$.
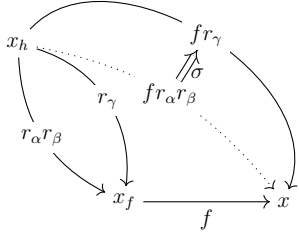
\[
\begin{tikzpicture}
[baseline=(c1.base),scale=0.7,
  every node/.style={transform shape}]
\node(c1) at (-1,2.5) {$x_h$};
\node(c2) at (1,-0.5) {$x_f$};
\node(c3) at (4,-0.5) {$x$};
\node(c4) at (1,-1.2) {} ;
\node(a1) at (2.5, -0.6) {};
\node(a2) at (3.5, 2.2) {};
\draw[->] (c1) to[bend right=30] node(f1){}
  node[midway, fill=white]{$r_{\alpha}r_{\beta}$} (c2);
\draw[->] (c1) to[bend left=40] node(f2){}
  node[midway, fill=white]{$r_{\gamma}$} (c2);
\draw[->, dotted] (c1) to[bend left=20] node(f4){}
  node[midway, fill=white]{$fr_{\alpha}r_{\beta}$} (c3);
\draw[->] (c1) to[bend left=80] node(f5){}
  node[midway, fill=white]{$fr_{\gamma}$} (c3);
\draw[double,double equal sign distance,-implies,
  shorten <=3, shorten >=3] (f4) to[bend right=1]
  node[right, pos=0.3]{$\tiny{\sigma}$} (f5);
%\draw[double,double equal sign distance,-implies, shorten <=3, shorten >=3] (f1) to[bend right=1]  node[above, pos=0.3]{$\tiny{\delta}$} (f2);
\draw[->] (c2) to node(g1){} node[below]{$f$} (c3);
\end{tikzpicture}
\]
Since $f$ is cartesian,
Remark~\ref{rem:explicit-material-of-weak-cart}
\ref{item:uniqueness-of-explicit-material-of-weak-cart} yields a unique vertical iso-2-cell
$\tau_{\alpha, \beta} \maps r_{\alpha} r_{\beta}
  \To  r_{\gamma}$
such that $f \dot \tau_{\alpha, \beta} = \sigma$.
Thus,
$(\beta \alpha) \oo (f \dot \tau_{\alpha, \beta})
  = \beta \oo (\alpha \dot r_{\beta})$.

For condition~\ref{third-axiom-fib}, if $\alubar = \id$,
then $\alpha$ is both cartesian and vertical, and hence an isomorphism.
Now we can use Remark~\ref{rem:explicit-material-of-weak-cart}\ref{item:uniqueness-of-explicit-material-of-weak-cart}
with $\alpha^{-1}$ for $\sigma$ and an identity for $\delubar$
to get $\delta\maps 1_{x_f}\To r_\alpha$ as well as an inverse for it.
It has the property required in \ref{third-axiom-fib}.

Now we prove condition~\ref{fourth-axiom-fib}, using the notation there,
and we wish to define the isomorphism in the left hand square.
We find we have two cartesian lifts of $\alubar'\dot\kubar$ to $gk_g$.
The first is the pasting
\[
  \pi^{-1} \alpha' (\rho\dot r_{\alpha'})
    \maps f k_f r_{\alpha'} \To g k_g
  \text{.}
\]
This is cartesian by Lemma~\ref{lem:closure-properties-of-cart-1,2-cells}\ref{lem:item:closed-under-composition},\ref{lem:item:iso-2-cell},
being composed of isomorphisms and the cartesian $\alpha'$.
The second is $\alpha\dot k_g$,
cartesian because $\alpha$ is cartesian and,
according to \ref{def:JfvB:whiskering-cart-2-cell},
its whiskering with any $1$-cell is cartesian.
These two cartesian lifts must be isomorphic,
so we get a unique iso-2-cell
between $f k_f r_{\alpha'}$ and $f r_\alpha k_g$,
over $\fubar\id_{\kubar}$,
that pastes with $\alpha$, $\rho$ and $\pi$ to give $\alpha'$.
Now we use Remark~\ref{rem:explicit-material-of-weak-cart}\ref{item:uniqueness-of-explicit-material-of-weak-cart}
to get a unique isomorphism in the left hand square of the diagram
with the required properties.

Finally, we shall prove \ref{fifth-axiom-fib},
which is similar to \ref{fourth-axiom-fib}.
Assume vertical $1$-cells $v_0$ and $v_1$ and a $2$-cell $\alpha_0$ over $\alubar$ as in the hypothesis of axiom \ref{fifth-axiom-fib}.
We use the cartesian property of the $2$-cell $\alpha \dot v_1$ to get a unique vertical $2$-cell
$\lambda \maps fv_0 \To f r_{\alpha} v_1$ such that $(\alpha \dot v_1) \oo \lambda = \alpha_0$.
By the cartesian structure of the $1$-cell $f$,
we can factor $\lambda$  as $f \dot \mu$ for a unique vertical $2$-cell $\mu$ with $f \dot \mu = \lambda$. Hence,  $(\alpha \dot v_1) \oo (f \dot \mu) = \alpha_0$.
\end{proof}

%%%%%%%%
%%%%%%%%
\section{Main results}
\label{sec:main-results}
%%%%%%%

We are now at a stage that we can state our main theorem.
Notice how our reformulation of the Johnstone criterion
assists our proof.
We do not have to deal with so many bipullback toposes,
and there is a single elementary topos $\qobar$ where
we examine models of the various contexts.

\begin{lem} \label{lem:main-thm}
  Let $U\maps \thT_1 \to \thT_0$ be an extension map of contexts with
  the fibration property in the Chevalley style (Definition~\ref{defn:Chevalley-fibration}),
  let $M$ be a model of $\thT_0$ in an elementary  topos $\CS$,
  and let $p\maps \CS[\thT_1/M] \to \CS$ be the classifier for $\thT_1/M$
  with generic model $G$.
  Suppose $f,g\maps q \toto p$ are two 1-cells in $\GTopos$
  and $\alpha\maps f \To g$ a 2-cell.
  We write $\varphi\eqdef\str{\alpha}(G,M)$,
  so that $\fiobar = \str{\alpha}G$ is a model of $\arrw{\thT_1}$ in $\qobar$.
  
  Then $\alpha$ is a cartesian 2-cell (in $\GTopos$ over $\ETopos$)
  iff $\eta_{\fiobar}$ is an isomorphism,
  where $(\eta_{\fiobar},\id)$ is the unit for
  $\model{\Gamma_1}{\qobar} \dashv \model{\Lambda_1}{\qobar}$.
\end{lem}
\begin{proof}
  ($\Rightarrow$):
  Let $\Nobar$ be the domain of $\fiobar\dot\Gamma_1\dot\Lambda_1$,
  and let $\Nubar \eqdef \str{q}\str{\fubar}M$.
  Then (see diagram~\ref{diag:context-transform})
  \[
  \begin{split}
    \Nobar\dot U &= \fiobar\dot\Gamma_1\dot\Lambda_1\dot\Gamma_0\dot\pi_0\dot U
      = \fiobar\dot\Gamma_1\dot\Lambda_1\dot\arrw{U}\dot\dom
      = \fiobar\dot\Gamma_1\dot\Lambda_1\dot\Gamma_1\dot U_1\dot\dom
    \\
      &= \fiobar\dot\Gamma_1\dot U_1\dot\dom
      = \fiobar\dot\arrw{U}\dot\dom
      = \fiobar\dot\dom\dot U
      = (\str{f}G)\dot U
      = \Nubar
      \text{,}
  \end{split}
  \]
  and so $N\eqdef(\Nobar,\Nubar)$ is a model of $U$ in $q$.
  \[
    \begin{tikzcd}
      \str{f}G
        \ar[->, rr, "\fiobar\eqdef\str{\alpha}G"]
        \ar[->, d, "\eta_{\fiobar}"]
        \ar[->, dotted, dd, bend right=40, "\str{\beta'}G" description]
      && \str{g}G
        \ar[-, double, d, ""]
      \\
      \Nobar
        \ar[->, rr, "\fiobar\dot\Gamma_1\dot\Lambda_1"]
        \ar[->, dotted, d, "e^{-}"]
      && \str{g}G
        \ar[-, double, d, ""]
      \\
      \str{e}G
        \ar[->, dotted, rr, "\str{\gamma}G" below]
        \ar[->, dotted, uu, bend left=80, "\str{\beta}G"]
      && \str{g}G
    \end{tikzcd}
  \]
  
  By the classifier property of $p$ (Proposition~\ref{prop:classifier}),
  and taking $\eubar\eqdef\fubar$ and $e_{-}\eqdef\id\maps\Nubar=\str{\fubar}M$,
  we obtain $e\maps q\to p$ and $(e^{-},\id)\maps N \iso \str{e}(G,M)$.
  Now by Proposition~\ref{prop:non-iso-downstairs} we get a unique $\gamma\maps e \To g$
  over $\gamubar\eqdef\alubar$ such that
  $\fiobar\dot\Gamma_1\dot\Lambda_1 = (\str{\gamma}G)e^{-}$.
  Again by Proposition~\ref{prop:non-iso-downstairs} we get a unique $\beta'\maps f \To e$
  over $\id_{\fubar}$ such that
  $e^{-}\eta_{\fiobar} = \str{\beta'}G$,
  and since $(\str{\gamma}G)(\str{\beta'}G) = \str{\alpha}G$
  it follows that $\gamma\beta'=\alpha$.
  
  By cartesianness of $\alpha$ we also have a unique $\beta\maps e\To f$ over $\id_{\fubar}$
  such that $\gamma = \alpha\beta$,
  and since $\alpha\beta\beta'=\gamma\beta'=\alpha$
  it follows that $\beta\beta'=\id_f$.
  We deduce that $(\str{\beta}G)e^{-}\eta_{\fiobar}=\id_{\str{f}G}$.
  
  Finally $\eta_{\fiobar}(\str{\beta}G)e^{-} = \id_{\Nobar}$
  follows from the adjunction $\Gamma_1\dashv\Lambda_1$,
  because both sides reduce by $\Gamma_1$ to the identity.
  Hence $\eta_{\fiobar}$ is an isomorphism,
  with inverse $(\str{\beta}G)e^{-}$.
  
  ($\Leftarrow$):
  Let $e\maps q\to p$ with $\gamma\maps e\To f$ such that $\gamubar=\alubar\betubar$.
  \[
    \begin{tikzcd}
      \str{f}G
        \ar[->, rr, "\fiobar\eqdef\str{\alpha}G"]
        \ar[->, d, "\eta_{\fiobar}"]
      && \str{g}G
        \ar[-, double, d, ""]
      \\
      \Nobar
        \ar[->, rr, "\fiobar\dot\Gamma_1\dot\Lambda_1"]
      && \str{g}G
        \ar[-, double, d, ""]
      \\
      \str{e}G
        \ar[->, dotted, rr, "\str{\gamma}G" below]
        \ar[->, dotted, u, "\psiobar" right]
        \ar[->, dotted, uu, bend left, "\str{\beta}G"]
      && \str{g}G
    \end{tikzcd}
    \quad\quad
    \begin{tikzcd}
      \str{\fubar}M
        \ar[->, rr, "\fiubar\eqdef\str{\alubar}M"]
        \ar[-, d, double]
      && \str{\gubar}M
        \ar[-, double, d, ""]
      \\
      \Nubar
        \ar[->, rr, "\fiubar\eqdef\str{\alubar}M"]
      && \str{\gubar}M
        \ar[-, double, d, ""]
      \\
      \str{\eubar}M
        \ar[->, dotted, rr, "\str{\gamubar}M" below]
        \ar[->, dotted, u, "\str{\betubar}M" right]
        \ar[->, dotted, uu, bend left, "\str{\betubar}M"]
      && \str{\gubar}M
    \end{tikzcd}
  \]
  
  By the adjunction $\Gamma_1\dashv\Lambda_1$ there is a unique $\thT_1$-morphism
  $\psiobar\maps \str{e}G\to\Nobar$ over $\str{\betubar}M$
  such that $(\fiobar\dot\Gamma_1\dot\Lambda_1)\psiobar = \str{\gamma}G$.
  Because $\eta_{\fiobar}$ is an isomorphism this corresponds to a unique
  $\psiobar'\maps \str{e}G\to\str{f}G$ over $\str{\betubar}M$ such that
  $\fiobar \psiobar'=\str{\gamma}G$.
  By Proposition~\ref{prop:non-iso-downstairs} this corresponds to a unique
  $\beta\maps e \To f$ over $\betubar$ such that
  $(\str{\alpha}G) (\str{\beta}G)=\str{\gamma}G$,
  \ie unique such that $\alpha\beta=\gamma$.
  This proves that $\alpha$ is cartesian.
\end{proof}

\begin{thm}
\label{thm:main}
If $U\maps \thT_1 \to \thT_0$ is an extension map of contexts with (op)fibration property
in the Chevalley style (Definition~\ref{defn:Chevalley-fibration}),
and $M$ a model of $\thT_0$ in an elementary  topos $\CS$,
then $p\maps \CS[\thT_1/M] \to \CS$ is an (op)fibration in the $2$-category $\ETopos$
in the Johnstone style (Definition~\ref{defn:Johnstone-2-fib}).
\end{thm}
\begin{proof}
Here we only prove the theorem for the case of fibrations. A proof for the opfibration case is similarly constructed.
According to Proposition~\ref{pro:Johnstone-fib-as-cod-fibrant-objects},
in order to establish that $p$ is a fibration in the 2-category $\ETopos$,
we have to verify that conditions \ref{def:JfvB:cart-lift-1-cell}-\ref{def:JfvB:whiskering-cart-2-cell}
in Definition~\ref{def:Johnstone-fib-via-Buckley} hold for $P=\ccod\maps\KK_\CD \to \KK$,
where $\KK=\ETopos$,
$\CD$ is the class of bounded geometric morphisms,
and so $\KK_{\CD}$ is $\GTopos$.

By Proposition~\ref{pro:cart-1-cell-aka-bipullback},
condition \ref{def:JfvB:cart-lift-1-cell} follows from the fact that
$p$ is bicarrable.

To prove condition \ref{def:JfvB:cart-lift-2-cell},
let $q\maps\qobar\to\qubar$ be a bounded geometric morphism,
let $g\maps q \to p$ be a 1-cell in $\KK_\CD$,
let $\fubar \maps \qubar \to \CS$  be geometric morphism
and $\alubar \maps \fubar \To \gubar$
a geometric transformation.
\[
\begin{tikzpicture}[baseline=(c2.base),scale=1,every node/.style={transform shape}]
\node(c0) at (0,1.5) {$\qobar$};
\node(c1) at (3,1.5) {$\CS[\thT_1/M]$};
\node(c2) at (0,-1.5) {$\qubar$};
\node(c3) at (3,-1.5) {$\CS$};
\node(d2) at (1.5,0.4) {$\gtdar$};
\node(e1) at (1.3,-2) {};
\node(e2) at (1.7,-1.2) {};
\node(e3) at (1.3,1) {};
\node(e4) at (1.7,1.8) {};
\draw[->] (c2) to[bend right=40] node(f1){} node[below]{$\fubar$} (c3);
\draw[->] (c2) to[bend left=30] node(g1){} node[above]{$\gubar$} (c3);
\draw[->] (c0) to node(y){} node[left]{$q$} (c2);
\draw[->] (c0) to[bend left=30] node(g2){} node[above]{$\gobar$} (c1);
%\draw[->, dashed] (c0) to[bend right=30] node(g3){} node[above]{$\gobar$} (c1);
\draw[->] (c1) to node(x){} node[right]{$p$} (c3);
\draw[double,double equal sign distance,-implies, shorten <=2, shorten >=2] (e1) to node[left]{$\alubar$} (e2);
\end{tikzpicture}
\]

We seek $f$ over $\fubar$ with a cartesian lift
$\alpha \maps f \To g$ of $\alubar$.
Notice that for the given model $M$ of $\thT_0$ in $\CS$, the component $M$ of the natural transformation $\alubar$
gives us a morphism
$\str{\alubar}M\maps \str{\fubar}M \to \str{\gubar}M$
of $\thT_0$-models in $\qubar$,
hence a $\arrw{\thT_0}$-model in $\qubar$.
Let us  write it as $\fiubar\maps\Nubar_f\to\Nubar_g$.
Then $\str{q}\fiubar$ is a model of $\arrw{\thT_0}$ in $\qobar$.

Let $G$ be the generic model of $\thT_1/M$ in
$\CS[\thT_1/M]$,
so that $(G, M)$ is a model of $U$ in $p$.
Hence we get $(\Nobar_g,\Nubar_g)\eqdef \str{g}(G,M)$
a model of $U$ in $q$, and
\[
  \fg \eqdef ( \Nobar_g, \str{q} \fiubar )
    \in \model{\cod^*(\thT_1)}{\qobar}
    \text{.}
\]
Then $\fg\dot\Lambda_1$
(see diagram \ref{diag:context-transform})
is a model $\fiobar\maps \Nobar_f \to \Nobar_g$
of $\arrw{\thT_1}$ in $\qobar$,
with $\Nobar_f = \fg \dot (\Lambda_1 ; \Gamma_0; \pi_0)$.
We also see that
$\fiobar \dot \arrw{U}
  = \fg \dot (\Lambda_1 ; \arrw{U})
  = \str{q}\fiubar$,
so $\varphi\eqdef(\fiobar,\fiubar)\maps N_f\to N_g$
is a homomorphism of $U$-models in $q$.
\[
  \begin{tikzcd}
    \Nobar_f
      \ar[r, "\fiobar"]
      \ar[d, maps to]
    & \Nobar_g
      \ar[d, maps to]
    \\
    \str{q}\Nubar_f
      \ar[r, "\str{q}\fiubar"]
    & \str{q}\Nubar_g
  \end{tikzcd}
\]

Thus we get two objects $(q,N_f)$ and $(q,N_g)$ of $\PQE$
together with $\varphi$ as in Proposition~\ref{prop:non-iso-downstairs}.
In addition we have $(p,(G,M))$,
and a $P$-cartesian 1-cell
\[
  (g, (\id\maps \Nobar_g = \str{g}G, \id\maps \Nubar = \str{\gubar}M))
  \maps (q,N_g) \to (p,(G,M))
  \text{.}
\]
By the classifier property we can also find a $P$-cartesian 1-cell
\[
  (f, (f^{-}, f_{-}))
  \maps (q,N_f) \to (p,(G,M))
  \text{.}
\]
We can now apply Proposition~\ref{prop:non-iso-downstairs} to find a 2-cell
$\alpha\maps f \To g$ over $\alubar$ that gives us $\fiobar$.

Since $\fiobar$ is defined to be of the form $\fg\dot\Lambda_1$,
so $\fiobar\dot\Gamma_1\dot\Lambda_1 = \fiobar$,
we find that $\eta_{\fiobar}$ is the identity
and $\eta_{\str{\alpha}G}$ is an isomorphism.
It follows from Lemma~\ref{lem:main-thm} that $\alpha$ is cartesian.

For proving \ref{def:JfvB:whiskering-cart-2-cell},
suppose we have $f,g\maps q \toto p$ and a cartesian 2-cell $\alpha\maps f\To g$.
By Lemma~\ref{lem:main-thm}, $\eta_{\str{\alpha}G}$ is an isomorphism.
Take any $1$-cell $k\maps q' \to q$ in $\GTopos$ where $q' \maps \qobar' \to \qubar'$.
Relative to the isomorphism of models
$\str{k} (\fg \dot \Lambda_1) \iso (\str{k} \fg) \dot \Lambda_1$,
$\str{k}$ preserves the unit $\eta$,
and so $\eta_{\str{k}\str{\alpha}G}$ is an isomorphism and,
by Lemma~\ref{lem:main-thm}, $\alpha\dot k$ is cartesian.

\end{proof}

The result can now be applied to the examples in \textsection\ref{sec:strict-internal-fibrations-in-2cat}.

\begin{enumerate}[label=(\roman*)]
\item
  The classifiers for Example~\ref{ex:torsors} are, by Diaconescu's theorem,
  those bounded geometric morphisms got as $[\mathbb{C},\CS]\to\CS$ for $\mathbb{C}$
  an internal category in $\CS$.
  Example~\ref{ex:torsorsOpfib} now tells us that such
  geometric morphisms are opfibrations in $\ETopos$.
  This is already known, of course, and appears
  in~\cite[B4.4.9]{johnstone:elephant1}.
  Note, however, that our calculation to prove the
  opfibration property in $\Con$ is elementary in nature.
  The proof of~\cite{johnstone:elephant1} verifies that the class of all such geometric morphisms
  satisfies the ``covariant tensor condition'',
  and such a technique cannot work for AUs as
  it uses the direct image parts of geometric morphisms.
\item
  The classifiers for Example~\ref{ex:pointed-set} are the local
  homeomorphisms.
  Their opfibrational character follows simply
  from our results, though note that it can also
  be deduced as a special case of the torsor result.
\item
  By Example~\ref{ex:SpecLFib},
  the classifiers for Example~\ref{ex:SpecL} are fibrations.
  Since the spectra of distributive lattices
  correspond to propositional coherent theories,
  this fibrational nature is already known
  from~\cite[B4.4.11]{johnstone:elephant1},
  which says that any coherent topos is a fibration.
  It will be interesting to see how far our methods can cover this general result.
\end{enumerate}

%%%%%%%%%
%%%%%%%%%
\section{Concluding thoughts}
%%%%%%%%%

What we have shown in this paper is that an important and extensive class of fibrations/opfibrations in $2$-category $\ETopos$ of toposes arises from strict fibrations/opfibrations in $2$-category $\Con$ of contexts. There are several advantages: first, the structure of strict fibrations/opfibrations in $\Con$ is much easier to study because of explicit and combinatorial description of $\Con$ and in particular due to existence of comma objects in there. Second, proofs concerning properties of based-toposes arising from $\Con$ are very economical since one only needs to work with strict models of contexts. Not only does this approach help us to avoid taking the pain of working with limits and colimits in $\ETopos / \CS$ and bookkeeping of coherence issues arising in this way, but it also gives us insights in inner working of $2$-categorical aspects of toposes via more concrete and constructive approach of contexts buildings and context extensions.

There is also an advantage from foundational point of view; for any $\CS$-topos $\CE$, there are logical properties internal to $\CE$ which are determined by internal logic of $\CS$. A consequence of this work is that  we can reason in $2$-category of contexts to get uniform results about toposes independent of their base $\CS$.

We hope that in the future work we can investigate the question that how much of 2-categorical structure of $\ETopos$ can be presented by contexts, and more importantly whether we can find simpler proofs in $\Con$ that can be transported to toposes.

%%%%%%%%%
%%%%%%%%%
\section{Acknowledgements}
%%%%%%%%%
The ideas and result of this paper were presented by the first author in PSSL 101 in University of Leeds (16-17 September 2017) and later in the conference Toposes in Como (27-29 June 2018). He wishes to thank the organizers of these events, particularly Nicola Gambino and Olivia Caramello, for the great opportunity to outline and discuss some of the ideas of this paper in those venues.

%%%%%%%%%%%%%%%%%%%%%%%%%%%%%%%%%%%%%%%%%%%%%%%%%%%%%%%%%
%%%%%%%%%%%%%%%%%%%%%%%%%%%%%%%%%%%%%%%%%%%%%%%%%%%%%%%%%
%%%%%%%%%%%%%%%%%%%%%%%%%%%%%%%%%%%%%%%%%%%%%%%%%%%%%%%%%
%%%%%%%%%%%%%%%%%%%%%%%%%%%%%%%%%%%%%%%%%%%%%%%%%%%%%%%%%
%
% END DOCUMENT
%
%%%%%%%%%%%%%%%%%%%%%%%%%%%%%%%%%%%%%%%%%%%%%%%%%%%%%%%%%
%%%%%%%%%%%%%%%%%%%%%%%%%%%%%%%%%%%%%%%%%%%%%%%%%%%%%%%%%
%%%%%%%%%%%%%%%%%%%%%%%%%%%%%%%%%%%%%%%%%%%%%%%%%%%%%%%%%
%%%%%%%%%%%%%%%%%%%%%%%%%%%%%%%%%%%%%%%%%%%%%%%%%%%%%%%%%

%\printbibliography
\bibliography{Refs}
\end{document}